\documentclass[a4paper,10pt]{amsart}

\usepackage[english]{babel}

\usepackage{verbatim}

\usepackage{amssymb} 

\usepackage{tikz, tikz-3dplot} 

\usepackage{lmodern} 

\usepackage{fancyvrb} 

\usepackage{url} 

\RequirePackage{color}
\definecolor{myred}{rgb}{0.75,0,0}
\definecolor{mygreen}{rgb}{0,0.5,0}
\definecolor{myblue}{rgb}{0,0,0.65}

\RequirePackage{ifpdf}
\ifpdf
  \IfFileExists{pdfsync.sty}{\RequirePackage{pdfsync}}{}
  \RequirePackage[pdftex,
   colorlinks = true,
   urlcolor = myblue, 
   citecolor = mygreen, 
   linkcolor = myred, 
   bookmarksopen=true]{hyperref}
\else
  \RequirePackage[hypertex]{hyperref}
\fi

\newtheorem{thm}{Theorem}[section]
\newtheorem{trm}[thm]{Theorem}
\newtheorem{cor}[thm]{Corollary}
\newtheorem{prop}[thm]{Proposition}
\newtheorem{lem}[thm]{Lemma}
\newtheorem{lma}[thm]{Lemma}

\theoremstyle{definition}
\newtheorem{defn}[thm]{Definition}
\newtheorem{defi}[thm]{Definition}

\newtheorem{ex}[thm]{Example}

\theoremstyle{remark}
\newtheorem{rem}[thm]{Remark}

\makeatletter
\let\c@equation\c@thm
\makeatother
\numberwithin{equation}{section}
\makeatletter
\let\c@figure\c@thm
\makeatother
\numberwithin{figure}{section}

\newcommand{\C}{\mathbb{C}}
\newcommand{\Q}{\mathbb{Q}}

\newcommand{\Z}{\mathbb{Z}}
\newcommand{\N}{\mathbb{N}}

\newcommand{\f}{\textbf}
\newcommand{\mr}{\mathrm}
\newcommand{\mi}{\mathit}
\newcommand{\mb}{\mathbf}
\newcommand{\bdim}{\mathbf{dim}\,}
\newcommand{\fa}{ \ \mathrm{for}\ \mathrm{all}\ }
\newcommand{\Hom}{\mathrm{Hom}}

\title[Permutation actions on Quiver Grassmannians and GKM-Theory]{Permutation actions on Quiver Grassmannians for the equioriented cycle via GKM-Theory}
\author{Martina Lanini and Alexander P\"{u}tz}

\begin{document}

\begin{abstract}In previous work we equipped quiver Grassmannians for nilpotent representations of the equioriented cycle with an action of an algebraic torus. We show here that the equivariant cohomology ring is acted upon by a product of symmetric groups and we investigate this permutation action via GKM techniques. In the case of (type A) flag varieties, or Schubert varieties therein, we recover Tymoczko's results on permutation representations.
\end{abstract}

\maketitle
\section*{Introduction}Realising a representation by geometric methods is in general very convenient, given the fruitful interplay between geometry and representation theory. For example, a geometric action of an algebraic group can induce an action of the corresponding Weyl group on cohomology, as in the flag variety case \cite{Tymoczko2008}. Nevertheless, there are cases in which the cohomology is equipped with a Weyl group action, which does not come from a geometric action on the variety, as for the Springer fibres \cite{Springer} or the Schubert varieties \cite{Tymoczko2008}. Further varieties whose cohomology has a structure of a Weyl group representation are Hessenberg varieties \cite{Tymoczko2008b}. Their investigation is currently a very active research area for combinatorialists, geometers and representation theorists. For instance, in type ${\tt A}$, the  permutation representation on Hessenberg varieties for regular semisimple elements is related by the Shareshian-Wachs Conjecture \cite{shareshianwachs}  (proven independently in \cite{BrosnanChow} and \cite{GuayPaquet}) to chromatic quasisymmetric functions of a certain graph.

The main inspiration for the present paper is Tymoczko's work (\cite{Tymoczko2008,Tymoczko2008b}), where symmetric group actions on Schubert varieties are concretely defined and investigated via equivariant localisation. More precisely, if a variety is nicely acted upon by an algebraic torus, its equivarant cohomology is encoded in the one-skeleton (the moment graph) of the action. This approach is referred to as GKM-Theory (after Goresky, Kottwitz and MacPherson \cite{GKM1998}). 

In \cite[Questions 5.8 \& 5.10]{Tymoczko2008b} the author asks: 

{
\begin{center}
    \emph{Can more geometric representations be realised using GKM?}
    
\vspace{2mm}    
    \emph{Can equivariant cohomology be used to construct natural families of twisted group representations?}
    
\end{center}
}
The aim of this paper is to positively answer these questions by showing that quiver Grassmannians for nilpotent representations of the equioriented cycle provides examples of varieties whose cohomology can be described via GKM-theory (in which case, we refer to them as GKM-varieties), and is equipped with an action of a product of appropriate symmetric groups.

In \cite{LaPu2020} we started a programme whose aim was to extend the use of GKM-theory to quiver Grassmannians. More precisely, we dealt with nilpotent representations for the equioriented cycle $\Delta_n$ and showed that the corresponding quiver Grassmannians can be equipped with an action of a torus $T$ which turns them into GKM-varieties. Since the main point of GKM-theory is to translate questions concerning cohomology into graph combinatorics, it was central to obtain a combinatorial description of the moment graph. Such a description relies on the combinatorics of the coefficient quiver.

In the present paper we address the third item of our wish list from \cite[Introduction]{LaPu2020} and take the challenge from \cite[Questions 5.8 \& 5.10]{Tymoczko2008b}: we show indeed that the torus equivariant cohomology of the quiver Grassmannian $\mr{Gr}_{\bf e}(M)$ of a nilpotent $\Delta_n$-representation $M$ admits an action of a product of symmetric groups  coming from a geometric action on the variety (Proposition \ref{propn:GeometricPermutation-Action}). We denote this reflection group by $\mathfrak{S}_{\underline{k}}$ (see \S \ref{sec:Permutation-Action} for the precise definition). In the flag variety case, this is nothing but Tymoczko's action from \cite{Tymoczko2008}, studied earlier via geometric methods by Brion in \cite{Brion2000}. 

In order to decompose the torus equivariant cohomology into irreducible $\mathfrak{S}_{\underline{k}}$-representations, we use an appropriate basis. Indeed, under a homogeneity assumption on $M$ (Definition \ref{def:homog-forest}), the equivariant cohomology of $\mr{Gr}_{\bf e}(M)$ admits a (unique) Knutson-Tao basis (KT-basis) as a module over $H_T^\bullet (\mr{pt})$ (Theorem \ref{trm:homg-forest-have-unique-KT-basis}). Existence and uniqueness of such a basis for equivariant cohomology modules is an interesting question in general, answered positively in  \cite{Tymoczko2008} for Schubert varieties and in \cite{GuilleminZara2003} for a wide class of smooth GKM-varieties. We want to point out that for the existence of the basis constructed in \cite{LaPu2020} we did not require homogeneity, but such a basis is harder to determine in general, given that it relies on the computation of equivariant Euler classes.

Thanks to the good combinatorial control that we have on the moment graph, we can investigate the effect of the permutations on the KT-basis. This is a direct generalisation of the (left) permutation action on (localised) equivariant Schubert classes described in \cite{Tymoczko2008}. We exploit this to deduce our main theorem: in the homogeneous case the permutation representation decomposes as a direct sum of (twisted) $\mathfrak{S}_{\underline{k}}$-trivial representations (see Theorem \ref{thm:mainThmTrivialRepn} for the precise statement).

How the permutation representation decomposes in the non homogeneous case remains an open question, which is certainly worth it to be investigated.

The last section of the paper is devoted to (left) divided difference operators, whose relevance in the classical setting is, for example, discussed in \cite{Tymoczko2009}. We show that they equip the equivariant cohomology with a structure of a graded module over a certain nil Hecke ring (Theorem \ref{thm:NHRingRepn}). 

While the behaviour of our permutation representation was exactly the same as in the classical setting, for the nil Hecke ring ${}^0\mathcal{H}$ the module structure presents new features. For example, the representation is not faithful in general and it is not a cyclic ${}^0\mathcal{H}$-module. It would be certainly interesting to further study the ${}^0\mathcal{H}$-module structure and provide, for example, conditions under which the module is cyclic.

Several questions remain to be addressed. Having proven the existence of a nice $H_T^\bullet(\textrm{pt})$-basis for the equivariant cohomology, it is natural to ask what it is possible to say about structure constants. This is one of the central (still open) problems of Schubert Calculus, so that we do not expect to be able to provide a complete answer. Nevertheless, it is tempting to conjecture that  an analogue of the equivariant Pieri rule exists also in our generality. We have made no effort yet in this direction, but we believe that the machinery we developed here provides useful tools to be applied  in this direction. 

\subsection*{Structure of the paper} Section~\ref{sec:Quiver-Grass} presents background material on Quiver representations, Quiver Grassmannians and $\mathbb{C}^*$-actions on them. In  Section~\ref{sec:GKM-Attractive-Forests}, after recalling basics of GKM-Theory, we focus on the GKM-variety structure of quiver Grassmannians for nilpotent representations of the equioriented cycle and provide a description of their moment graphs. Section~\ref{sec:comb-basis} is about existence and uniqueness of KT-basis for the equivariant cohomology of the above quiver Grassmannians. In Section~\ref{sec:Permutation-Action} we introduce the permutation representation on quiver Grassmannianns for nilpotent $\Delta_n$-representations and describe the effect on the moment graph. We investigate the permutation representation via their action on KT-classes in Section~\ref{sec:perm-action-KT-classes}. Finally, Section~\ref{sec:DividedDifferenceOps} is about divided difference operators and nil Hecke ring action.


\subsection*{Acknowledgements}
 We thank Francesco Esposito for useful discussions, and Arun Ram for helping with references. We acknowledge the PRIN2017 CUP E8419000480006, and the MIUR Excellence Department Project awarded to the Department of Mathematics, University of Rome Tor Vergata, CUP
 E83C18000100006.

\section{Generalities on Quiver Grassmannians}\label{sec:Quiver-Grass}
A \f{quiver} $Q$ consists of a set of vertices $Q_0$ and a set of oriented edges $Q_1$ between the vertices. For every arrow $(a: i \to j) \in Q_1$, let $s_a=i, t_a=j \in Q_0$ denote the source and the target of an arrow, respectively. 
 
2

A $Q$\f{-representation} $M$ is a pair of tuples $(M^{(i)})_{i \in Q_0}$ and $(M_a)_{a \in Q_1}$, where the $M^{(i)}$ are $\C$-vector spaces and each $M_a$ is a linear map from $M^{(s_a)}$ to $M^{(t_a)}$. 
By $V_{Q}$ we denote the $Q$-representation with $V_{Q}^{(i)}=\C$ for all $i \in Q_0$ and $V_{Q,a}= \mr{id}_\C$ for all $a \in Q_1$.
If the sum of all $\dim_\C M^{(i)}$ is finite, the $Q$-representation $M$ is said to be finite-dimensional. By $\mr{rep}_\C(Q)$ we denote the category of finite-dimensional $Q$-representations. The \f{support of} $M$ is the full subquiver of $Q$ parametrised by the vertices $i \in Q_0$ such that $\dim_\C M^{(i)} \neq 0$. 

For two $Q$-representations $M$ and $N$, a $Q$\f{-morphism} $f$ is a tuple of linear maps $f_i : M^{(i)} \to N^{(i)}$ such that $f_{t_a} \circ M_a = N_a \circ f_{s_a} $. The set of all $Q$-morphisms from $M$ to $N$ is denoted by $\Hom_Q(M,N)$. A $Q$-representation $U$ is called \f{subrepresentation} of $M$ if there exists an injective $Q$-morphism from $U$ to $M$. Or equivalently the vector spaces $U^{(i)}$ are subspaces in the $M^{(i)}$ and $M_a U^{(s_a)} \subseteq U^{(t_a)}$ holds for all $a \in Q_1$. The \f{dimension vector} of a $Q$-representation $U$ is defined as
\[
\bdim U := \big( \dim_\C U^{(i)} \big)_{i \in Q_0} \in \Z^{Q_0}.
\]
\begin{defi}\label{def:quiver-grass}
The \f{quiver Grassmannian} $\mr{Gr}_\mb{e}(M)$ is the variety of all subrepresentations $U$ of $M$ whose dimension vector equals $\mb{e}\in \Z^{Q_0}$.
\end{defi}
\begin{rem}\label{rem:dim-vector-of-subrep}
From now on we assume the choice of a dimension vector $\mb{e} \in \Z^{Q_0}$ such that $\mr{Gr}_\mb{e}(M)$ is non-empty. 
\end{rem}

\subsection{$\C^*$-Action on Quiver Representations and Fixed Points of Quiver Grassmannians}\label{subsec:C*-Action}
In this subsection we introduce $\C^*$-actions on quiver Grassmannians and describe the fixed points of this action for the case that the underlying weights are well behaved. A basis $B$ of $M \in \mr{rep}_\C(Q)$ consists of basis 
\[ B^{(i)} =  \big\{ v^{(i)}_k \vert k \in [m_i] \big\} \]
for each vector space $M^{(i)}$ of the $Q$-representation $M$, where $m_i := \dim_\C M^{(i)}$ for all $i \in Q_0$, and $[m]:=\{1,\dots,m\}$.
\begin{defi}
Let $M \in \mr{rep}_\C(Q)$ and $B$ a basis of $M$. The \f{coefficient quiver} $Q(M,B)$ consists of: 
\begin{itemize}
\item[(QM0)] the vertex set  $Q(M,B)_0=B$,
\item[(QM1)] the set of arrows $Q(M,B)_1$, containing $(\tilde{a}: v_k^{(i)} \to v_\ell^{(j)})$ if and only if $(a:i\to j) \in Q_1$ and the coefficient of $v_\ell^{(j)}$ in $M_a v_k^{(i)}$ is non-zero.
\end{itemize}
\end{defi}
\begin{rem}\label{rem:Segments-to-Indec}
Every $M \in \mr{rep}_\C (Q)$ is isomorphic to a direct sum of indecomposable representations and this decomposition is unique up to reordering \cite[Theorem~1.11]{Kirillov2016}. Hence, by definition of coefficient quivers, there exists a basis $B$ such that the connected components of $Q(M,B)$ are in bijection with the indecomposable summands of $M$. This implies that with a basis $B$, which satisfies this assumption, the image of every basis vector $v \in B^{(s_a)}$ under the map $M_a$ for $a \in Q_1$ is either zero or the unique basis vector $v' \in  B^{(t_a)}$. Note that in general there are several bases with this property and the different possibilities to choose this basis will play an important role in Section~\ref{subsec:P-S-moment graph}.
\end{rem}
\begin{defi}
Let $M \in \mr{rep}_\C(Q)$ and $B$ be a basis of $M$. 
\begin{itemize}
\item[(i)] A \f{grading} on $Q(M,B)_0$ is a tuple $\mb{wt} = \mr{wt}(b)_{b \in B} \in \Z^B$. 

\item[(ii)] $M$ is \f{well behaved} if for every arrow $a : i \to j$ of $Q$  and every element $b \in B^{(i)}$ there exists an element $b' \in B^{(j)}$ and $c \in \C$ (possibly zero) such that
\( M_a b = cb', \) 
and there exists a grading on $Q(M,B)_0$ so that:
\begin{itemize}
\item[(D1)] for all $i \in Q_0$ all vectors from $B^{(i)}$ have different degrees;
\item[(D2)] for every arrow $a : i \to j$ of $Q$, whenever $b_1 \neq b_2$ are elements of $B^{(i)}$ such that $M_a b_1 = c_1b_1'$ and $M_a b_2= c_2b_2'$ with $c_1,c_2 \in \C^*$ and $b_1',b_2' \in B^{(j)}$, we have:
\[ \mr{wt}(b_1') -   \mr{wt}(b_2')  =  \mr{wt}(b_1) -  \mr{wt}(b_2).\]
\end{itemize}
\end{itemize}
\end{defi}
\begin{rem}\label{rem:C-action}
A grading on $Q(M,B)_0$ induces a $\C^*$-action on the vector spaces of the $Q$-representation $M$, defined on the basis vectors as 
\begin{equation}\label{eqn:C-action}
     z \cdot b := z^{\mr{wt}(b)}b \quad \mr{for} \ z \in \C^*, \ b \in B.
\end{equation}
If $M$ is well behaved, the action extends to the quiver Grassmannian $\mr{Gr}_\mb{e} (M)$ by \cite[Lemma~1.1]{Cerulli2011}.
\end{rem}
Let $\mr{Gr}_\mb{e}(M)^{\C^*}$ denote the fixed point set of the $\C^*$-action on the quiver Grassmannian. The following theorem was proven by Cerulli Irelli in \cite[Theorem~1]{Cerulli2011}.
\begin{trm}\label{trm:Fixed-Points}
Let $M$ be a well behaved $Q$-representation. Then 
\[ \mr{Gr}_\mb{e}(M)^{\C^*} = \big\{ N \in \mr{Gr}_\mb{e} (M) \  \big\vert \ N^{(i)} \ \mi{is} \ \mi{spanned} \ \mi{by} \ \mi{part} \ \mi{of} \ B^{(i)} \big\}. \]
\end{trm}
\begin{rem}\label{rem:coordinate-subreps}
Since the vector spaces $N^{(i)}$ of the points in $\mr{Gr}_\mb{e}(M)^{\C^*}$ are spanned by subsets of $B^{(i)}$, we refer to them as \f{coordinate subrepresentations}. 
\end{rem}
\begin{rem}
Theorem~\ref{trm:Fixed-Points} implies that the number of fixed points is finite. In this case it equals the Euler characteristic of the quiver Grassmannian \cite[Section~2]{Cerulli2011}. 
\end{rem} 
\begin{defi}\label{def:successor-closed}
A subquiver $L \subset Q(M,B)$ is \f{successor closed} if $L_1$ contains every arrow from $Q(M,B)_1$ starting in a vertex in $L_0 \subset Q(M,B)_0$. 
\end{defi}
The following equivalent formulation of Theorem~\ref{trm:Fixed-Points} is proven in \cite[Proposition~1]{Cerulli2011}. For a successor closed $L \subset Q(M,B)$ we write $L \overrightarrow{\subset} Q(M,B)$ for short, and identify it with the subrepresentation of $M$ whose vector spaces are spanned by $L_0 \cap B^{(i)}$ for all $i \in Q_0$.
\begin{cor}\label{cor:Fixed-Points-as-suc-closed-sub-quiv}
Let $M$ be a well behaved $Q$-representation. Then 
\[ \mr{Gr}_\mb{e}(M)^{\C^*} \cong \big\{ L \overrightarrow{\subset} Q(M,B) \  \big\vert \ \vert L_0 \cap B^{(i)}\vert =  e_i \fa i \in Q_0 \big\} =: SC_\mb{e}^Q(M). \]
\end{cor}
\subsection{Nilpotent Representations of the Equioriented Cycle}\label{subsec:nilpot-reps}
In this section, we introduce the class of quiver representations which we want to study. From now on, we restrict us to the setting that $Q$ is the equioriented cycle on $n$ vertices, which we denote by $\Delta_n$. The set of vertices and the set of arrows are both in bijection with $\Z_n := \Z/n\Z$.

\begin{defi}\label{def:nilpot-rep} A $\Delta_n$-representation $M$ is called \f{nilpotent}, if there exists a non-negative integer $N$ such that
\( M_{a+N} \circ  M_{a+N} \circ \dots \circ M_{a+1} \circ M_a = 0 \fa a \in \Z_n.\)
\end{defi}
\begin{ex}\label{ex:indec-nilpot}
Let $A_\ell$ be the equioriented quiver of type ${\tt A}$ with $l$-many vertices and for $i \in \Z_n$ let $F_i: A_\ell \to \Delta_n$ send $j \in [\ell] := \{ 1, \dots ,\ell\}$ to $i+j-1 \mod n$, and $(a:j \to j+1 \mod n)$ to $(\overline{a}: i+j-1 \mod n \to i+j \mod n)$. This induces the $\Delta_n$-representation $U_i(\ell):=F_i(V_{A_\ell})$.
\end{ex}
\begin{trm} (\cite[Theorem~7.6]{Kirillov2016}) The $\Delta_n$-representation $U_i(\ell)$ is nilpotent and indecomposable for every $i \in \Z_n$ and $\ell \in \Z_{\geq 1}$. All finite-dimensional nilpotent indecomposable  $\Delta_n$-representations are of this form.
\end{trm}
\subsection{Attractive Gradings}\label{subsec:Attractive Gradings}
We want to understand the local structure of the quiver Grassmannians to parametrise the one-dimensional $T$-orbits of a larger torus $T$. For this purpose we construct a cellular decomposition into the attracting sets of $\C^*$-fixed points. These attracting sets are isomorphic to affine spaces only under some assumptions about the grading:

\begin{defi}\label{def:attractive-grading}
A grading $\mb{wt}:=\mr{wt}(b)_{b \in B} \in \Z^B$ on $Q(M,B)_0$ is \f{attractive } if:
\begin{itemize}
\item[(AG1)] for any $i\in \Z_n$ it holds that $\mr{wt}(v^{(i)}_k)>\mr{wt}(v^{(i)}_\ell)$ whenever $k>\ell$,
\item[(AG2)] for any $i \in \Z_n$ there exists a weight $d(i)\in\Z$ such that \[\mr{wt}\big(v^{(i+1)}_\ell\big)=\mr{wt}\big(v^{(i)}_k\big)+d(i)\] whenever $v^{(i)}_k\to v^{(i+1)}_\ell\in Q(M,B)_1$.
\end{itemize}
\end{defi} 
\begin{prop}\label{prop:attractive-grading}(\cite[Proposition~5.1]{LaPu2020}) Every nilpotent $\Delta_n$-representation admits an attractive grading.
\end{prop}

It turns out, that in the construction of cellular decompositions of quiver Grassmannians we need one additional property of the nilpotent representations:
\begin{defi} A nilpotent $M \in  \mr{rep}_\C (\Delta_n)$ is \f{alignable} if there exists a basis $B$, such that for $Q(M,B)$ the following holds over each $i \in Z_n$:
\begin{itemize}
\item[(SA1)] endpoints of segments have larger indices than points with outgoing arrows:
if $M_i v^{(i)}_\ell = 0$ and $M_i v^{(i)}_k \neq 0$, then $ \ell > k$. 
\item[(SA2)] outgoing arrows are order preserving:\\
if $M_i v^{(i)}_\ell =  v^{(i+1)}_{\ell'}$ and $M_i v^{(i)}_k = v^{(i+1)}_{k'}$ with $\ell>k$, then $ \ell' > k'$. 
\end{itemize}
\end{defi}
\begin{lma}\label{lma:2s-free-alignable}(\cite[Proposition~4.8]{LaPu2020})
Every nilpotent $M \in  \mr{rep}_\C (\Delta_n)$ is alignable.
\end{lma}

\begin{rem}\label{rem:alignability}
From now on we can assume without loss of generality that we have a basis of a nilpotent $M \in  \mr{rep}_\C (\Delta_n)$ such that its coefficient quiver is aligned. 
\end{rem}
\begin{rem}\label{rem:different-alignments}
The aligned coefficient quiver obtained in \cite[Proposition~4.8]{LaPu2020} is not unique. In fact, it will be useful to work with different alignments. This is studied in more detail in Section~\ref{sec:comb-basis}.
\end{rem}
\subsection{Cellular Decomposition of Quiver Grassmannians for Nilpotent Representations}\label{subsec:Cellular-Decomposition-nilpot}
In this section we recall the construction of cellular decompositions of quiver Grassmannians for nilpotent representations. Let $X$ be a projective variety and let $\C^*$ act on $X$ with finitely many fixed points. Let $\{x_1, \ldots, x_m\}$ be the fixed point set, which we denote by $X^{\C^*}$. This induces a decomposition
\begin{equation}\label{eqn:BBdecomposition}
X=\bigcup_{i\in [m]} W_i, \quad\hbox{ with }\quad   W_i := \left\{ x \in X \mid \lim_{z \to 0} z.x =x_i \right\},
\end{equation}
where $W_i$ is called attractive locus of $x_i$. We call this a \f{BB-decomposition} since decompositions of this type where first studied by Bialynicki-Birula in \cite{Birula1973}. 

\begin{trm}\label{trm:cell_decomp-forests}(\cite[Theorem~5.6]{LaPu2020})
Let $M \in \mr{rep}_\C(\Delta_n)$ be nilpotent and let $\C^*$ act on $\mr{Gr}_\mb{e}(M)$ induced by an attractive grading on the aligned coefficient quiver $Q(M,B)$ as in Definition~\ref{def:attractive-grading}. Then, for every fixed point $L \in \mr{Gr}_{\mb{e}}(M)^{\C^*}$, the attractive locus $W_L$
is isomorphic to an affine space and the quiver Grassmannian admits a cellular decomposition 
\[  
\mr{Gr}_{\mb{e}}(M) = \bigcup_{L \in \mr{Gr}_{\mb{e}}(M)^{\C^*}} W_L.  
\]
\end{trm} 

\begin{rem}\label{rem:bad-alignment}
Note that the way in which the connected components of the coefficient quiver are arranged plays an important role. For example, consider the quiver Grassmannian $\mr{Gr}_\mb{e}(M)$ which is isomorphic to the Feigin degeneration of the flag variety $\mathcal{F}l_3$ \cite[Proposition~2.7]{CFR2012}. With the basis as in \cite[Remark~3.14]{CFR2013} $M$ is alignable, and the attractive loci of the fixed points are isomorphic to affine spaces. The representation $M$ is not aligned if we fix $B$ such that 
\begin{center}
\begin{tikzpicture}[scale=0.35]
\node at (-3.3,2.5) {$Q(M,B) = $};
\draw[fill=black] (0,3) circle (.12);
\draw[arrows={-angle 90}, shorten >=2, shorten <=2]  (0,3) -- (1.5,3);
\draw[fill=black] (1.5,3) circle (.12);
\draw[fill=black] (0,2) circle (.12);
\draw[arrows={-angle 90}, shorten >=2, shorten <=2]  (0,2) -- (1.5,2);
\draw[fill=black] (1.5,2) circle (.12);
\draw[fill=black] (0,4) circle (.12);
\draw[fill=black] (1.5,1) circle (.12);
\end{tikzpicture}
\end{center}
Nevertheless there still exists an attractive grading, but the attractive loci for the corresponding $\C^*$-action are not isomorphic to affine spaces: Consider the attractive locus of the $\C^*$-fixed point 
\begin{center}
\begin{tikzpicture}[scale=0.35]
\node at (-2.9,2.5) {$L= $};
\draw[fill=white] (0,3) circle (.12);
\draw[arrows={-angle 90}, shorten >=2, shorten <=2]  (0,3) -- (1.5,3);
\draw[fill=black] (1.5,3) circle (.12);
\draw[fill=white] (0,2) circle (.12);
\draw[arrows={-angle 90}, shorten >=2, shorten <=2]  (0,2) -- (1.5,2);
\draw[fill=black] (1.5,2) circle (.12);
\draw[fill=black] (0,4) circle (.12);
\draw[fill=white] (1.5,1) circle (.12);
\end{tikzpicture}
\end{center}
For simplicity we write $v_j := v^{(1)}_j$ for $j \in [3]$ and $w_j := v^{(2)}_j$ for $j \in [4]$  then 
\[W_L = \Big\{ p = \big(<v_1+a v_2 + bv_3>,<w_2+c w_4,w_3+d w_4>\big) \, \vert \, ac +bd = 0 \Big\}. \]
\end{rem}

\section{GKM-Theory}\label{sec:GKM-Attractive-Forests}
We explain here how moment graphs encode the equivariant cohomology of GKM-varieties by a version of the Localisation Theorem for equivariant cohomology from \cite{GKM1998}. We recall moreover that by \cite{LaPu2020} every quiver Grassmannian for a nilpotent $\Delta_n$-representation is a GKM-variety whose moment graph has an explicit combinatorial description.
\subsection{Basics of GKM-Theory}\label{subsec:Basics-GKM-Theory} 
Let $X$ be a projective algebraic variety over $\C$. The action of an algebraic torus $T \cong (\C^*)^r$ on $X$ is \f{skeletal} if the number of $T$-fixed points and the number of one-dimensional $T$-orbits in $X$ is finite. A cocharacter $\chi \in \mathfrak{X}_*(T)$ is called \f{generic} for the $T$-action on $X$ if $X^T = X^{\chi(\C^*)}$. By $H_T^\bullet(X)$ we denote the $T$-equivariant cohomology of $X$ with rational coefficients. 
\begin{defi} The pair $(X,T)$ is a \f{GKM-variety} if the $T$-action on $X$ is skeletal and the rational cohomology of $X$ vanishes in odd degrees.
\end{defi}
\begin{rem}
By \cite[Lemma~2]{Brion2000} this is equivalent to \cite[Definition~1.4]{LaPu2020}. 
\end{rem}
For every one-dimensional $T$-orbit $E$ in a projective GKM-variety there exists a $T$-equivariant isomorphism between its closure $\overline{E}$ and $\C\mathbb{P}^1$. This implies that each one-dimensional $T$-orbit connects two distinct $T$-fixed points of $X$.
\begin{defi}\label{def:moment-graph}Let $(G,T)$ be a GKM-variety, and let $\chi\in \mathfrak{X}_*(T)$ be a generic cocharacter. The corresponding \f{moment graph} $\mathcal{G}=\mathcal{G}(X,T, \chi)$ of a GKM-variety is given by the following data:
\begin{itemize}
\item[(MG0)] the $T$-fixed points as vertices, i.e.: $\mathcal{G}_0 = X^T$,
\item[(MG1)] the closures of one-dimensional $T$-orbits $\overline{E} = E \cup \{x,y\}$ as edges in $\mathcal{G}_1$, oriented from $x$ to $y$ if $\lim_{\lambda \to 0}\chi(\lambda).p=x$ for $p \in E$,
\item[(MG2)] every $\overline{E}$ is labelled by a character $\alpha_E \in \mathfrak{X}^*(T)$ describing the $T$-action on $E$. 
\end{itemize}
\end{defi}
\begin{rem}\label{rem:torus-characters-part-i}
The characters in (MG2) are uniquely determined up to a sign but this sign has no effect on the statement of Theorem \ref{thm:GKM}. We can hence assume that a choice of a sign for any edge label has been made once and for all.
\end{rem}
\begin{rem}
In \cite{LaPu2020}, the torus cocharacter from (MG1) was implicit in the definition of moment graph, and did not appear in the notation, because, to achieve our goals, it was enough to pick any generic cocharacter which provided a cellular decomposition. In the present paper, the orientation of the edges has a fundamental role, in order to study the structure of permutation representation on the equivariant cohomology. We will see that not any generic cocharacter  is good for our purposes, even in the case it provides a cellular decomposition (cf. Example~\ref{ex:cohomology-generators-loop-quiver}). 
\end{rem}
\begin{rem}\label{rem:torus-characters-part-ii} Let $T$ be a torus of rank $q$ and let $\{\tau_1, \ldots, \tau_q\}$ be a $\Z$-basis  of its character lattice $\mathfrak{X}^*(T)$. For any torus character $\alpha$, by abuse of notation, we denote by $\alpha$ also its image $\alpha\otimes 1$ in the $\Q$-vector space $\mathfrak{X}^*(T)\otimes_\Z \Q$.
Following \cite[Section~2]{Tymoczko2008}, we identify the symmetric algebra of this vector space with the polynomial ring $\Q[T]=\Q[\tau_1,\dots,\tau_q]$, and hence the latter  with $R:=H_T^\bullet(\textrm{pt})$.  Observe that $R$ is $\mathbb{Z}$-graded with $\textrm{deg}(\tau_i)=2$ for any $i$.
\end{rem}
One of our main motivation to consider moment graphs is the following result: Goresky-Kottwitz-MacPherson version of the Localisation Theorem.
\begin{trm}\label{thm:GKM}(\cite{GKM1998}) Let $(X, T)$ be a GKM-variety with moment graph $\mathcal{G}=\mathcal{G}(X,T,\chi)$. Then 
\[
H_T^\bullet(X) \cong \left\{(f_x)\in\bigoplus_{x\in \mathcal{G}_0}R \ \Big| \
\begin{array}{c}
f_{x_E}-f_{y_E}\in \alpha_{E} R\\
 \hbox{ for any }\overline{E}=E\cup\{x_E, y_E\}\in\mathcal{G}_1\end{array}
\right\}.\]  
\end{trm}
\subsection{Torus Actions}\label{subsec:T-Action}
Let $M$ be a nilpotent $\Delta_n$-representation. In \cite{LaPu2020}, we defined an action of the torus $T:= (\C^*)^{d+1}$ on $\mr{Gr}_{\mb{e}}(M)$, where $d$ is the number of indecomposable summands of $M$. By \cite[Theorem~1.11]{Kirillov2016}, we can always assume a choice of the basis $B$ of $M$ such that the connected components of $Q(M,B)$ are in bijection with the indecomposable summands of $M$. Notice that this property is invariant under permutations of the basis elements. 

Fix an order $U_1, \dots, U_{d}$ for the indecomposables. This induces an associated partition of the basis $B$, into subsets $B_U$ corresponding to the points on a given indecomposable $U$. As introduced in Example~\ref{ex:indec-nilpot}, every indecomposable summand $U$ of $M$ has a unique initial and terminal vertex. The starting point gets index zero. Let $\ell_j:=\ell(U_j)$ be the number of points on the $j$-th indecomposable in $Q(M,B)$ for $j \in [d]$, then we have 
\begin{equation}\label{eqn:indec-basis}
 B_{U_j} = \{ b_{j,0}, \dots, b_{j,\ell_j-1} \}.
\end{equation}
For any $\gamma := (\gamma_0, (\gamma_j)_{j \in [d]}) \in T$ we set
\[ \gamma.b_{j,p} := \gamma^{p}_0 \gamma_j \cdot b_{j,p}.\]
By linear extension, we get an action on the vector space $\bigoplus_{i\in \Z/n\Z} M^{(i)}$, which preserves each summand. 
\begin{lem}\label{lma:T-action-extends-to-quiver-Grass}(\cite[Lemma~5.10]{LaPu2020} Let $M$ be a representation of $\Delta_n$, and let $T$ act on $\bigoplus M^{(i)}$ as above. Then for any $N\in\mr{Gr}_{\mb{e}}(M)$ and any $\gamma\in T$, $\gamma\cdot N\in \mr{Gr}_\mb{e}(N)$. 
\end{lem}

In Section~\ref{subsec:Quiver-GKM-Theorem}, we recall a class of quiver Grassmannians which, together with the $T$-action defined above, have the structure of GKM-varieties.  
\begin{prop}\label{prop:generic-cochar}
Let $M \in \mr{rep}_\C(\Delta_n)$ be nilpotent with $d$-many indecomposable direct summands. Let the torus $T:= (\C^*)^{d+1}$ act on $\mr{Gr}_\mb{e}(M)$ as in Lemma~\ref{lma:T-action-extends-to-quiver-Grass}. Consider a $\C^*$-action on $\mr{Gr}_\mb{e}(M)$ induced by an attractive grading as in \eqref{eqn:C-action}. Then there exists a generic cocharacter such that the above $\C^*$-action coincides with the one obtained by composing $\chi$ with the $T$-action.
\end{prop}
\begin{proof}
Let $\C^*$ act on $\mr{Gr}_\mb{e}(M)$ with the weights of the attractive grading. The corresponding cocharacter $\chi$ is constructed analogous to the proof of \cite[Theorem~5.12]{LaPu2020}, where it is constructed for one explicit attractive grading. 
\end{proof}

\subsection{Quiver Grassmannians of Nilpotent Representations of the Cycle are GKM-Varieties}\label{subsec:Quiver-GKM-Theorem}
As promised, we recall here that quiver Grassmannians for nilpotent $\Delta_n$-representations admit a structure of BB-filterable GKM-varieties.
\begin{defi}\label{def:rational-cell}
We say that $W_i$ from \eqref{eqn:BBdecomposition} is a \f{rational cell} if it is rationally smooth at all $w\in W_i$. This in turn holds if 
\[H^{2\textrm{dim}_{\C}(W_i)}(W_i, W_i\setminus \{w\})\simeq \Q  \quad\hbox{ and  }\quad  H^m(W_i,W_i\setminus\{w\})=0 \]
for any $m\neq 2 \textrm{dim}_{\C}(W_i)$. 
\end{defi}
\begin{defi}\label{def:BB-filterable}
A projective $T$-variety $X$ is \f{BB-filterable} if: 
\begin{enumerate}
\item[(BB1)] the fixed point set $X^T$ is finite,
\item[(BB2)] there exists a generic cocharacter $\chi: \C^* \rightarrow T$, i.e. $X^{\chi(\C^*)} = X^T$, such that 
 the associated BB-decomposition consists of rational cells. 
\end{enumerate}
\end{defi}
\begin{trm}\label{trm:attr-forests-are-GKM}(\cite[Theorem~6.5]{LaPu2020}) Let $M \in \mr{rep}_\C(\Delta_n)$ be nilpotent with $d$-many indecomposable summands, and let $T:= (\C^*)^{d+1}$ act on $\mr{Gr}_\mb{e}(M)$ as in Lemma~\ref{lma:T-action-extends-to-quiver-Grass}. 
Then $(\mr{Gr}_\mb{e}(M),T)$ is a projective BB-filterable GKM-variety.
\end{trm}

\subsection{Moment Graph}\label{subsec:Moment-Graph}
Let $M \in \mr{rep}_\C(\Delta_n)$ be nilpotent. Take a $\C^*$-action on $\mr{Gr}_\mb{e}(M)$ corresponding to an attractive grading on $Q(M,B)$ as in Definition~\ref{def:attractive-grading}.
\begin{defi}\label{defn:Mutation-Relation}Let $S \in \mr{SC}_\mb{e}^{\Delta_n}(M)$ parametrise a $\C^*$-fixed point as in Corollary~\ref{cor:Fixed-Points-as-suc-closed-sub-quiv}. A connected predecessor closed subquiver of $S$ is called {\bf  movable part}. We say $S,H \in \mr{SC}_\mb{e}^{\Delta_n}(M)$ are \f{mutation related} if they differ by the position of exactly one movable part. 
\end{defi}
The $\C^*$-weights naturally provide an orientation of the mutation relations:
\begin{defi}\label{defn:FundMutation}
For a fixed $\C^*$-action, we order the vertices in the coefficient quiver of $M$ increasingly by their weight. If for mutation related $S,H \in \mr{SC}_\mb{e}^{\Delta_n}(M)$, the indices of the vertices on the movable part as subset of $H \subset Q(M,B)$ are larger than for $S$ we say that $H$ is obtained from $S$ by a \f{fundamental mutation}.
\end{defi}

For $(\gamma)=(\gamma_0, \gamma_1, \ldots, \gamma_{d})\in T$, we define 
\[
\epsilon_j:T\rightarrow \C^*, \quad (\gamma_0,\gamma_1, \ldots, \gamma_{d})\mapsto \gamma_j \qquad (j \in [d])
\] 
and 
\[
\delta:T\rightarrow \C^*, \quad (\gamma_0,\gamma_1, \ldots, \gamma_{d})\mapsto \gamma_0.
\] 
\begin{rem}\label{rem:T-weights-one-dim-orbits}
The $T$-weight of a mutation $\mu$ is computed as the weight difference of the terminal vertex of the moved subsegment before and after the movement: By construction of the $T$-action, all other points of the moved subsegment have the same $T$-weight difference. Let $p \in B$ be the terminal vertex before the mutation and $p' =\mu(p) \in B$ the terminal vertex after the mutation. From the definition of the mutations it follows that there are unique $j_s$ and $j_t$ in $[d]$ with $j_s < j_t$ such that $p$ lives on the indecomposable summand $U_{j_s}$ and $p'$ lives on indecomposable summand $U_{j_t}$ of $M$, i.e. $p=b_{j_s,k}$ and $p=b_{j_t,k'}$ for some $k \in [\ell_{j_s}]$ and $k' \in [\ell_{j_t}]$. 

By \cite[Remark~6.10]{LaPu2020} and the proof of \cite[Theorem~6.13]{LaPu2020}, every mutation relation is identified with a unique one-dimensional $T$-orbit. Observe that this statement is independent of the cocharacter $\chi$ since it does not depend on the orientation of the corresponding edge in the moment graph. In particular, it does not require the orientation used in \cite{LaPu2020}. 

The orientation of the edges in the moment graph is obtained from the cocharacter $\chi$ corresponding to the attractive grading, as described in (MG2) of Definition~\ref{def:moment-graph}. By construction this coincides with the orientation of the mutation relations. The $T$-weight of the mutation $\mu$ is the character of the corresponding one-dimensional $T$-orbit and equals
\begin{equation}\label{eqn:Characters-in-moment-graph}
\alpha_\mu := \epsilon_{j_t}-\epsilon_{j_s}+  \big(k'-k\big)\delta.
\end{equation}
\end{rem}

\begin{trm}\label{trm:comb-moment-graph}
Let $M \in \mr{rep}_\C(\Delta_n)$ be nilpotent with $d$-many indecomposable summands and attractive grading and let $X:=\mr{Gr}_\mb{e}(M)$. Let $T:= (\C^*)^{d+1}$ act on $\mr{Gr}_\mb{e}(M)$ as in Proposition~\ref{prop:generic-cochar} and let $\chi$ be the corresponding cocharacter of the attractive grading. The vertices of the moment graph $\mathcal{G}(X,T,\chi)$ are in bijection with the successor closed subquivers in $SC_\mb{e}^{\Delta_n}(M)$. For $S,H \in SC_\mb{e}^{\Delta_n}(M)$ there exists an arrow in the moment graph from $S$ to $H$ if and only if there exists a fundamental mutation $\mu(S) =H$. The label of this edge is given by $\alpha_\mu$ as in (\ref{eqn:Characters-in-moment-graph}).
\end{trm}
\begin{proof}
The $T$-fixed points and the $\C^*$-fixed points coincide by Proposition~\ref{prop:generic-cochar}. Hence the description of the $T$-fixed points follows by Corollary~\ref{cor:Fixed-Points-as-suc-closed-sub-quiv}. The parametrisation of the edges and their labels are obtained as in Remark~\ref{rem:T-weights-one-dim-orbits}. 
\end{proof}
\begin{cor}\label{cor:mutations-provide-acyclic-orientation}
Fundamental mutations induce a partial order on $SC_\mb{e}^Q(M)$.
\end{cor}
\begin{proof}
It follows immediately from the orientation of the mutation relations that no sequence of fundamental mutations can build an oriented cycle in the moment graph.
\end{proof}
\begin{ex}\label{ex:not-P-S-preperation}
The quiver Grassmannian $\mr{Gr}_4(M)$ for the loop quiver $\Delta_1$, with  $M = A_4 \oplus A_2 \oplus A_2$ where $A_N \cong \C[t]/(t^N)$, together with the action by $T=(\C^*)^{3+1}$ as described in Lemma~\ref{lma:T-action-extends-to-quiver-Grass} is a BB-filterable GKM-variety. We set $U_1=A_4$, $U_2=A_2$ and $U_3=A_2$. Then we consider the grading induced by $\mr{wt}(a) = 3$, $\mr{wt}(b_{1,1})=1$, $\mr{wt}(b_{2,1})=8$ and $\mr{wt}(b_{3,1})=9$. Let $\chi$ be the corresponding cocharacter of $T$ as in Proposition~\ref{prop:generic-cochar}. Hence we can apply Theorem~\ref{trm:comb-moment-graph} to compute its moment graph $\mathcal{G}(X,T,\chi)$. There are $9$ $T$-fixed points:
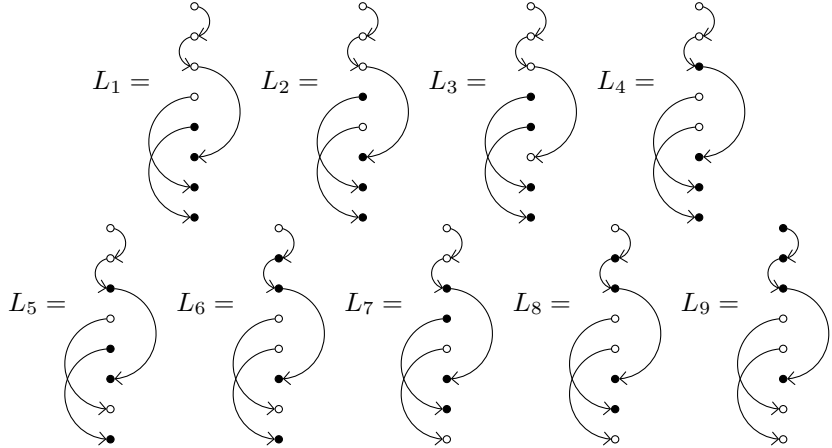
\begin{center}
\begin{tikzpicture}[scale=0.4]
\node at (-2.4,4.5) {$L_1 = $};
\draw[fill=white] (0,7) circle (.12);
\draw[fill=white] (0,6) circle (.12);
\draw[fill=white] (0,5) circle (.12);
\draw[fill=white] (0,4) circle (.12);
\draw[fill=black] (0,3) circle (.12);
\draw[fill=black] (0,2) circle (.12);
\draw[fill=black] (0,1) circle (.12);
\draw[fill=black] (0,0) circle (.12);

 \def\centerarc[#1](#2)(#3:#4:#5);%
    {
    \draw[#1]([shift=(#3:#5)]#2) arc (#3:#4:#5);
    }
\centerarc[arrows={-angle 90}](0,6.5)(76:-76:0.5cm); 
\centerarc[arrows={-angle 90}](0,5.5)(180-76:180+76:0.5cm); 
\centerarc[arrows={-angle 90}](0,3.5)(85:-85:1.5cm); 
\centerarc[arrows={-angle 90}](0,2.5)(180-85:180+85:1.5cm); 
\centerarc[arrows={-angle 90}](0,1.5)(180-85:180+85:1.5cm); 
\end{tikzpicture}
\begin{tikzpicture}[scale=0.4]
\node at (-2.4,4.5) {$L_2 = $};
\draw[fill=white] (0,7) circle (.12);
\draw[fill=white] (0,6) circle (.12);
\draw[fill=white] (0,5) circle (.12);
\draw[fill=black] (0,4) circle (.12);
\draw[fill=white] (0,3) circle (.12);
\draw[fill=black] (0,2) circle (.12);
\draw[fill=black] (0,1) circle (.12);
\draw[fill=black] (0,0) circle (.12);

 \def\centerarc[#1](#2)(#3:#4:#5);%
    {
    \draw[#1]([shift=(#3:#5)]#2) arc (#3:#4:#5);
    }
\centerarc[arrows={-angle 90}](0,6.5)(76:-76:0.5cm); 
\centerarc[arrows={-angle 90}](0,5.5)(180-76:180+76:0.5cm); 
\centerarc[arrows={-angle 90}](0,3.5)(85:-85:1.5cm); 
\centerarc[arrows={-angle 90}](0,2.5)(180-85:180+85:1.5cm); 
\centerarc[arrows={-angle 90}](0,1.5)(180-85:180+85:1.5cm); 
\end{tikzpicture}
\begin{tikzpicture}[scale=0.4]
\node at (-2.4,4.5) {$L_3 = $};
\draw[fill=white] (0,7) circle (.12);
\draw[fill=white] (0,6) circle (.12);
\draw[fill=white] (0,5) circle (.12);
\draw[fill=black] (0,4) circle (.12);
\draw[fill=black] (0,3) circle (.12);
\draw[fill=white] (0,2) circle (.12);
\draw[fill=black] (0,1) circle (.12);
\draw[fill=black] (0,0) circle (.12);

 \def\centerarc[#1](#2)(#3:#4:#5);%
    {
    \draw[#1]([shift=(#3:#5)]#2) arc (#3:#4:#5);
    }
\centerarc[arrows={-angle 90}](0,6.5)(76:-76:0.5cm); 
\centerarc[arrows={-angle 90}](0,5.5)(180-76:180+76:0.5cm); 
\centerarc[arrows={-angle 90}](0,3.5)(85:-85:1.5cm); 
\centerarc[arrows={-angle 90}](0,2.5)(180-85:180+85:1.5cm); 
\centerarc[arrows={-angle 90}](0,1.5)(180-85:180+85:1.5cm); 
\end{tikzpicture}
\begin{tikzpicture}[scale=0.4]
\node at (-2.4,4.5) {$L_4 = $};
\draw[fill=white] (0,7) circle (.12);
\draw[fill=white] (0,6) circle (.12);
\draw[fill=black] (0,5) circle (.12);
\draw[fill=white] (0,4) circle (.12);
\draw[fill=white] (0,3) circle (.12);
\draw[fill=black] (0,2) circle (.12);
\draw[fill=black] (0,1) circle (.12);
\draw[fill=black] (0,0) circle (.12);

 \def\centerarc[#1](#2)(#3:#4:#5);%
    {
    \draw[#1]([shift=(#3:#5)]#2) arc (#3:#4:#5);
    }
\centerarc[arrows={-angle 90}](0,6.5)(76:-76:0.5cm); 
\centerarc[arrows={-angle 90}](0,5.5)(180-76:180+76:0.5cm); 
\centerarc[arrows={-angle 90}](0,3.5)(85:-85:1.5cm); 
\centerarc[arrows={-angle 90}](0,2.5)(180-85:180+85:1.5cm); 
\centerarc[arrows={-angle 90}](0,1.5)(180-85:180+85:1.5cm); 
\end{tikzpicture}

\begin{tikzpicture}[scale=0.4]
\node at (-2.4,4.5) {$L_5 = $};
\draw[fill=white] (0,7) circle (.12);
\draw[fill=white] (0,6) circle (.12);
\draw[fill=black] (0,5) circle (.12);
\draw[fill=white] (0,4) circle (.12);
\draw[fill=black] (0,3) circle (.12);
\draw[fill=black] (0,2) circle (.12);
\draw[fill=white] (0,1) circle (.12);
\draw[fill=black] (0,0) circle (.12);

 \def\centerarc[#1](#2)(#3:#4:#5);%
    {
    \draw[#1]([shift=(#3:#5)]#2) arc (#3:#4:#5);
    }
\centerarc[arrows={-angle 90}](0,6.5)(76:-76:0.5cm); 
\centerarc[arrows={-angle 90}](0,5.5)(180-76:180+76:0.5cm); 
\centerarc[arrows={-angle 90}](0,3.5)(85:-85:1.5cm); 
\centerarc[arrows={-angle 90}](0,2.5)(180-85:180+85:1.5cm); 
\centerarc[arrows={-angle 90}](0,1.5)(180-85:180+85:1.5cm); 
\end{tikzpicture}
\begin{tikzpicture}[scale=0.4]
\node at (-2.4,4.5) {$L_6 = $};
\draw[fill=white] (0,7) circle (.12);
\draw[fill=black] (0,6) circle (.12);
\draw[fill=black] (0,5) circle (.12);
\draw[fill=white] (0,4) circle (.12);
\draw[fill=white] (0,3) circle (.12);
\draw[fill=black] (0,2) circle (.12);
\draw[fill=white] (0,1) circle (.12);
\draw[fill=black] (0,0) circle (.12);

 \def\centerarc[#1](#2)(#3:#4:#5);%
    {
    \draw[#1]([shift=(#3:#5)]#2) arc (#3:#4:#5);
    }
\centerarc[arrows={-angle 90}](0,6.5)(76:-76:0.5cm); 
\centerarc[arrows={-angle 90}](0,5.5)(180-76:180+76:0.5cm); 
\centerarc[arrows={-angle 90}](0,3.5)(85:-85:1.5cm); 
\centerarc[arrows={-angle 90}](0,2.5)(180-85:180+85:1.5cm); 
\centerarc[arrows={-angle 90}](0,1.5)(180-85:180+85:1.5cm); 
\end{tikzpicture}
\begin{tikzpicture}[scale=0.4]
\node at (-2.4,4.5) {$L_7 = $};
\draw[fill=white] (0,7) circle (.12);
\draw[fill=white] (0,6) circle (.12);
\draw[fill=black] (0,5) circle (.12);
\draw[fill=black] (0,4) circle (.12);
\draw[fill=white] (0,3) circle (.12);
\draw[fill=black] (0,2) circle (.12);
\draw[fill=black] (0,1) circle (.12);
\draw[fill=white] (0,0) circle (.12);

 \def\centerarc[#1](#2)(#3:#4:#5);%
    {
    \draw[#1]([shift=(#3:#5)]#2) arc (#3:#4:#5);
    }
\centerarc[arrows={-angle 90}](0,6.5)(76:-76:0.5cm); 
\centerarc[arrows={-angle 90}](0,5.5)(180-76:180+76:0.5cm); 
\centerarc[arrows={-angle 90}](0,3.5)(85:-85:1.5cm); 
\centerarc[arrows={-angle 90}](0,2.5)(180-85:180+85:1.5cm); 
\centerarc[arrows={-angle 90}](0,1.5)(180-85:180+85:1.5cm); 
\end{tikzpicture}
\begin{tikzpicture}[scale=0.4]
\node at (-2.4,4.5) {$L_8 = $};
\draw[fill=white] (0,7) circle (.12);
\draw[fill=black] (0,6) circle (.12);
\draw[fill=black] (0,5) circle (.12);
\draw[fill=white] (0,4) circle (.12);
\draw[fill=white] (0,3) circle (.12);
\draw[fill=black] (0,2) circle (.12);
\draw[fill=black] (0,1) circle (.12);
\draw[fill=white] (0,0) circle (.12);

 \def\centerarc[#1](#2)(#3:#4:#5);%
    {
    \draw[#1]([shift=(#3:#5)]#2) arc (#3:#4:#5);
    }
\centerarc[arrows={-angle 90}](0,6.5)(76:-76:0.5cm); 
\centerarc[arrows={-angle 90}](0,5.5)(180-76:180+76:0.5cm); 
\centerarc[arrows={-angle 90}](0,3.5)(85:-85:1.5cm); 
\centerarc[arrows={-angle 90}](0,2.5)(180-85:180+85:1.5cm); 
\centerarc[arrows={-angle 90}](0,1.5)(180-85:180+85:1.5cm); 
\end{tikzpicture}
\begin{tikzpicture}[scale=0.4]
\node at (-2.4,4.5) {$L_9 = $};
\draw[fill=black] (0,7) circle (.12);
\draw[fill=black] (0,6) circle (.12);
\draw[fill=black] (0,5) circle (.12);
\draw[fill=white] (0,4) circle (.12);
\draw[fill=white] (0,3) circle (.12);
\draw[fill=black] (0,2) circle (.12);
\draw[fill=white] (0,1) circle (.12);
\draw[fill=white] (0,0) circle (.12);

 \def\centerarc[#1](#2)(#3:#4:#5);%
    {
    \draw[#1]([shift=(#3:#5)]#2) arc (#3:#4:#5);
    }
\centerarc[arrows={-angle 90}](0,6.5)(76:-76:0.5cm); 
\centerarc[arrows={-angle 90}](0,5.5)(180-76:180+76:0.5cm); 
\centerarc[arrows={-angle 90}](0,3.5)(85:-85:1.5cm); 
\centerarc[arrows={-angle 90}](0,2.5)(180-85:180+85:1.5cm); 
\centerarc[arrows={-angle 90}](0,1.5)(180-85:180+85:1.5cm); 
\end{tikzpicture}
\end{center}
The unlabelled moment graph computes as:
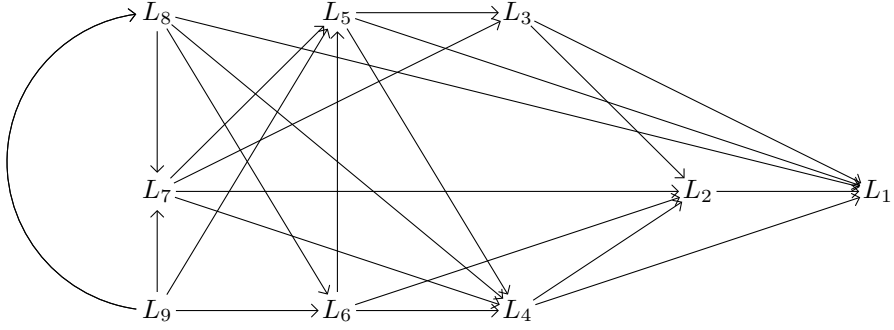
\begin{center}
\begin{tikzpicture}[scale=0.79]
\node at (0,0) {$L_1$};
\node at (-1*3,0) {$L_2$};
\node at (-2*3,3) {$L_3$};
\node at (-2*3,-2) {$L_4$};
\node at (-3*3,3) {$L_5$};
\node at (-3*3,-2) {$L_6$};
\node at (-4*3,3) {$L_8$};
\node at (-4*3,0) {$L_7$};
\node at (-4*3,-2) {$L_9$};

 \def\centerarc[#1](#2)(#3:#4:#5);%
    {
    \draw[#1]([shift=(#3:#5)]#2) arc (#3:#4:#5);
    }
\centerarc[arrows={-angle 90}](-4*3,0.5)(262:98:2.5cm);

\draw[arrows={-angle 90}, shorten >=7, shorten <=7]  (-4*3,-2) -- (-4*3,0); 
\centerarc[arrows={-angle 90}](-4*3,0.5)(262:98:2.5cm); 
\draw[arrows={-angle 90}, shorten >=7, shorten <=7]  (-4*3,-2) -- (-3*3,-2); 
\draw[arrows={-angle 90}, shorten >=7, shorten <=7]  (-4*3,-2) -- (-3*3,3); 

\draw[arrows={-angle 90}, shorten >=7, shorten <=7]  (-4*3,3) -- (-4*3,0); 
\draw[arrows={-angle 90}, shorten >=7, shorten <=7]  (-4*3,3) -- (-3*3,-2); 
\draw[arrows={-angle 90}, shorten >=7, shorten <=7]  (-4*3,3) -- (-2*3,-2); 
\draw[arrows={-angle 90}, shorten >=7, shorten <=7]  (-4*3,3) -- (0,0); 

\draw[arrows={-angle 90}, shorten >=7, shorten <=7]  (-4*3,0) -- (-3*3,3); 
\draw[arrows={-angle 90}, shorten >=7, shorten <=7]  (-4*3,0) -- (-2*3,-2); 
\draw[arrows={-angle 90}, shorten >=7, shorten <=7]  (-4*3,0) -- (-2*3,3); 
\draw[arrows={-angle 90}, shorten >=7, shorten <=7]  (-4*3,0) -- (-1*3,0); 

\draw[arrows={-angle 90}, shorten >=7, shorten <=7]  (-3*3,-2) -- (-3*3,3); 
\draw[arrows={-angle 90}, shorten >=7, shorten <=7]  (-3*3,-2) -- (-2*3,-2); 
\draw[arrows={-angle 90}, shorten >=7, shorten <=7]  (-3*3,-2) -- (-1*3,0); 
\draw[arrows={-angle 90}, shorten >=7, shorten <=7]  (-3*3,3) -- (-2*3,-2); 
\draw[arrows={-angle 90}, shorten >=7, shorten <=7]  (-3*3,3) -- (-2*3,3); 
\draw[arrows={-angle 90}, shorten >=7, shorten <=7]  (-3*3,3) -- (0,0); 
\draw[arrows={-angle 90}, shorten >=7, shorten <=7]  (-2*3,-2) -- (-1*3,0); 
\draw[arrows={-angle 90}, shorten >=7, shorten <=7]  (-2*3,-2) -- (0,0); 
\draw[arrows={-angle 90}, shorten >=7, shorten <=7]  (-2*3,3) -- (-1*3,0); 
\draw[arrows={-angle 90}, shorten >=7, shorten <=7]  (-2*3,3) -- (0,0); 
\draw[arrows={-angle 90}, shorten >=7, shorten <=7]  (-1*3,0) -- (0,0); 

\end{tikzpicture}
\end{center}
The label of the edge $L_9 \to L_8$ is given by $\epsilon_2-\epsilon_1+\delta$: $T$ acts on the terminal vertex of the moved subtree by $\gamma_1$ if we consider it as point in $L_9$. For this mutation, the moved subtree only consists of this vertex. As point in $L_8$ we have an action by $\gamma_0\gamma_2$. Then (\ref{eqn:Characters-in-moment-graph}) implies that the edge weight is $\epsilon_2-\epsilon_1+\delta$. Analogously, we compute the weight of all other edges.
\end{ex}
\section{Combinatorial Basis for the T-equivariant Cohomology}\label{sec:comb-basis}
Let $(X,T)$ be a GKM-variety and let $\chi \in \mathfrak{X}_*(T)$ be a generic cocharacter. We write $x \succeq_\chi y$ if there exists a directed path (possibly of length zero) from $x$ to $y$ in the moment graph $\mathcal{G}(X,T,\chi)$. The following is shown in \cite[Section~5]{Tymoczko2005}.
\begin{prop}\label{prop:acyclic-orientation}
Let $(X,T)$ be a GKM-variety. Then there exists a generic cocharacter $\chi \in \mathfrak{X}_*(T)$ such that $\succeq_\chi$ is a partial order on $X^T$, that is $\mathcal{G}(X,T, \chi)$ is acyclic.
\end{prop}

Given a vertex $x$ of an oriented graph $\mathcal{G}$, we define
\[
 \mathcal{G}_1^{\partial x}:=\{E\in \mathcal{G}_1\mid \exists\ y\in \mathcal{G}_0\hbox{ with }E:x\to y\}.
\]

The classes in the following definition are named after Knutson and Tao since they first used classes of this form to construct a basis for the equivariant cohomology of Grassmannians in \cite{KnutsonTao2003}.
\begin{defi}\label{def:K-T-class}(\cite[Definition 2.12]{Tymoczko2008})
Let $(X,T)$ be a GKM-variety with moment graph $\mathcal{G}=\mathcal{G}(X,T,\chi)$. A \f{Knutson-Tao class} for $x \in X^T$ is an equivariant class $p^x = (p^x_y)_{y \in X^T}\in H_T^\bullet(X)$ such that:
\begin{enumerate}
\item[(KT1)] \(p^x_x=  \prod_{\substack{ E\in \mathcal{G}_1^{\partial x}}} \alpha_E, \)
\item[(KT2)] each $p^x_y$ is a homogeneous polynomial in $\Q[T]$ with $\deg p^x_y = \deg p^x_x$, 
\item[(KT3)] $p^x_y = 0$ for each $y \in X^T$ such that $y \nsucceq_\chi x$.
\end{enumerate}
\end{defi}
\begin{rem}\label{rem:schubert-classes}
For (generalised) flag varieties, the Knutson-Tao classes are equivariant Schubert classes and the partial order $\succeq$ is the Bruhat order \cite[Proposition~4.6 and Proposition~4.7]{Tymoczko2008}.
\end{rem}
\begin{prop}$\mr{(}$\cite[Proposition~2.13]{Tymoczko2008}$\mr{).}$\label{prop:K-T-basis}
Let $(X,T)$ be a GKM-variety and let $\chi$ be a generic cocharacter such that $\mathcal{G}(X,T,\chi)$ is acyclic.
Suppose that for each $x \in X^T$ there exists at least one Knutson-Tao class $p^x \in H_T^\bullet(X)$. Then the classes $\{p^x \vert x \in X^T\}$ form a basis for $H_T^\bullet(X)$. 
\end{prop}
\begin{rem}
Unfortunately, in general it is not clear whether Knutson-Tao classes exist. Their existence is proven for several smooth GKM-varieties in \cite{GuilleminZara2003}, and Schubert varieties in \cite[\S3]{Tymoczko2008}. For BB-filterable GKM-varieties, we prove existence in Theorem ~\ref{thm:ExistenceKTBasis}. This applies to quiver Grassmannians of nilpotent $\Delta_n$-representations by Theorem \ref{trm:attr-forests-are-GKM}.
\end{rem}
\subsection{Existence of a KT-basis}
In this subsection we prove the existence of KT-classes for certain GKM-varieties.  

\begin{prop}\label{propn:GoodTotalOrder}Let $(X,T)$ be a BB-filterable GKM-variety and let $\chi$ be such that $\mathcal{G}=\mathcal{G}(X,T,\chi)$ contains no oriented cycles. Then, there exists a total order $\leq$ on $X^T$ such that
\begin{enumerate}
    \item[(TO1)] $\bigcup_{y\leq x} W_{y}$ is closed for any $x\in X^T$,
    
   \item[(TO2)] $x\succeq_\chi x'\quad\Rightarrow \quad x\geq x'$
\end{enumerate}
\end{prop}
\begin{proof}
By Proposition~\ref{prop:acyclic-orientation} there exists a cocharacter $\chi \in \mathfrak{X}_*(T)$ such that $\mathcal{G}(X,T,\chi)$ contains no oriented cycles. Moreover, we fixed $\succeq := \succeq_\chi$ for some $\chi$ with this property for the whole section. Thanks to \cite[Lemma~4.12]{Carrell02}, we know that a total order satisfying (TO1) exists. We show that also (TO2) holds under our hypothesis. 

The property (TO2) is trivial for $x=x'$. Thus, let $x,x'\in X^T$ be such that $x\succ x'$, that is there is at least one path in $\mathcal{G}$ from $x$ to $x'$. We prove the claim by induction on the length $r$ (i.e. number of arrows) of a minimal length path between them. If $x\rightarrow x'\in \mathcal{G}_1$, then there is a one-dimensional torus orbit $\mathcal{O}$ contained in $W_{x}$ whose closure contains the fixed point $x'\in X^T$. It follows that
\[
\overline{W_{x}}\cap W_{x'}\neq \emptyset.
\]
On the other hand, $\overline{W_{x}}\subseteq\bigcup_{y\leq x}W_y$ holds by property (TO1), and hence
\[
\left(\bigcup_{y\leq x} W_{y}\right)\cap W_{x'}\neq \emptyset.
\]
Since $W_y\cap W_{y'}\neq\emptyset$ if and only if $W_y=W_{y'}$, we conclude that 
\[
W_{x'}\subseteq \bigcup_{y\leq x} W_{y}
\]
and so $x'\leq x$.
Now, let $r\geq 2$, and let
\[
x=x_0\rightarrow x_1\rightarrow\ldots\rightarrow x_r=x'
\]
be a minimal length path. This implies that $x\succ x_1$, $x_1\succ x'$ and the length of a minimal path connecting $x_1$ and $x'$ is $<r$. By the base step we have that $x>x_1$, and by induction we get $x_1>x'$. We deduce $x>x'$.
\end{proof}
\begin{rem}\label{rem:total-order-fixed-points}
We assume now (and for the rest of the paper) that we have fixed a total order satisfying the properties in Proposition~\ref{propn:GoodTotalOrder}.
\end{rem}

\begin{lem}\label{lem:ExistenceKTbasis}Let $(X,T)$ be a BB-filterable GKM-variety and let $\chi$ be such that $\mathcal{G}=\mathcal{G}(X,T,\chi)$ contains no oriented cycles. Then, there exists an $R$-basis $\{q^x\}_{x\in X^T}$ of $H_T^\bullet(X)$ satisfying the following properties:
\begin{enumerate}
\item[(B1)] $q^x_x=\prod_{E\in\mathcal{G}^{\partial x}_1}\alpha_E$,
\item[(B2)] $q^x_y$ is homogeneous of degree $\textrm{deg}(q^x_x)$,
\item[(B3)] $q^x_y=0$ for any $y\leq x$.
\end{enumerate}
\end{lem}
\begin{proof} By \cite[Theorem~2.5]{LaPu2020} we know that there exists an $R$-basis $\{\widetilde{q}^x\}_{x\in X^T}$ of $H^\bullet_T(X)$ such that, for any $x\in X^T$, the following properties hold:
\begin{enumerate}
    \item[(P1)] $\widetilde{q}^x_x=\textrm{Eu}_T(x,W_x)$.
    \item[(P2)] $\widetilde{q}^x_y=0$ for any $y\leq x$,
\end{enumerate}
Here $\textrm{Eu}_T(x,W_x)$ denotes the equivariant Euler class and in our case we have  $\textrm{Eu}_T(x,W_x)=\prod_{E\in\mathcal{G}^{\partial x}_1}\alpha_E$ (cf. \cite[proof of Lemma~2.4]{LaPu2020}).

To prove the claim it remains to modify the given basis in such a way that any element is homogeneous. Since $H_T^\bullet(X)$ is graded, we can write 
\[
\widetilde{q}^x=\sum_{i\geq 0} \widetilde{q}^{x,i},
\]
where $\widetilde{q}^{x,i}$ denotes the homogeneous part of degree $2i$. We claim that we can just set $q^x:=\widetilde{q}^{x,\#\mathcal{G}^{\partial x}_1}$. 

Indeed, it satisfies (B1), (B2) and (B3). By property (B3), $\widetilde{q}^{x,i}_y= 0$ holds for any $y\leq x$ and any $i\geq 0$ (in particular for $i=\#\mathcal{G}^{\partial x}_1$). (B1) and (B2) hold by construction. It remains to check that $\{q^x\}$ is a basis. 

This follows if we can prove that there exist $c_{y}^x\in R$ such that
\[
q^x=\widetilde{q}^x-\sum_{y<x}c_y^x \widetilde{q}^y.
\]

We give a recursive construction of the $c_{y}^x$: Let $x\in X^T$ and $i\neq \#\mathcal{G}^{\partial x}_1$ be such that $\widetilde{q}^{x,i}\neq 0$. Assume that $y$ is minimal (w.r.t. the total order $\leq$) such that $\widetilde{q}^{x,i}_y\neq 0$. By property (P2), we have that $y\geq x$. Moreover, $i\neq \#\mathcal{G}^{\partial x}_1$ implies  $y\neq x$ by property (P1). Since $y$ is minimal, $\widetilde{q}^{x,i}_{y'}= 0$ holds for any $y'\in X^T$ with $y\rightarrow y'\in\mathcal{G}^{\partial y}_1$ and so $\prod_{E\in\mathcal{G}^{\partial y}_1}\alpha_E$ divides $\widetilde{q}^{x,i}_{y}$ by Theorem~\ref{thm:GKM}. We conclude that 
\[
\left(\widetilde{q}^{x,i}
- \frac{\widetilde{q}^{x,i}_y}{\prod_{E\in\mathcal{G}^{\partial y}_1}\alpha_E}\widetilde{q}^{y}\right)_{y}=0.
\]
If $\widetilde{q}^{x,i}-\left( \frac{\widetilde{q}^{x,i}_y}{\prod_{E\in\mathcal{G}^{\partial y}_1}\alpha_E}\right)\widetilde{q}^y\neq 0$ we look for the minimal non-zero entry, and proceed in this way to construct the $c_{y}^x\in R$ recursively.\end{proof}

\begin{thm}\label{thm:ExistenceKTBasis}Let $(X,T)$ be a BB-filterable GKM-variety and let $\chi$ be such that $\mathcal{G}=\mathcal{G}(X,T,\chi)$ contains no oriented cycles. 
Then, there exists a KT-basis of $H_T^\bullet(X)$ (w.r.t. the partial order $\succeq:=\succeq_\chi$).
\end{thm}
\begin{proof}
By Lemma \ref{lem:ExistenceKTbasis}, we know that there exists a basis $\{q^x\}_{x\in X^T}$ satisfying (KT1), (KT2) and such that
\[
q^x_{y}=0 \quad\hbox {for any }y\leq x.
\]
We want to modify this basis in such a way that also (KT3) holds. 

Since we have a total order on $X^T$, it is convenient to renumber its elements as
\[
x_1,x_2, \ldots, x_{m}
\]
with $m := \#X^T$, meaning that $x_i\leq x_j$ if and only if $i\leq j$. We will show by induction on $m-i$ that there exist homogeneous elements $c_j^i\in R$ such that
\[
p^{x_i}:=q^{x_i}-\sum_{j<i}c_j^i q^{x_j},
\]
with $\deg(c_j^i)+\deg(q^{x_j})=\#\mathcal{G}_1^{\partial x_i}$ and that $p^{x_i}$ verifies (KT1), (KT2) and (KT3).

If $i=m$, we can just take $p^{x_{m}}=q^{x_m}$, as $x_{m}$ must be a maximal element of the poset $(X^T,\preceq)$ by (TO2). We assume now that we have determined  $\{c_j^k\mid j<k\}$ such that the resulting  $p^{x_k}$ has properties (KT1), (KT2) and (KT3) for any $k>i$. If $q^{x_i}$ already verifies (KT3), we take $p^{x_i}:=q^{x_i}$. Otherwise there is a minimal $r$ (w.r.t. $\leq$) such that $q^{x_i}_{x_r}\neq 0$ and $x_r\not\succeq x_i$. Note that if $y\in X^T$ is such that $x_r\rightarrow y$, then $y\preceq x_r$ and by (TO2) we have $y<x_r$. Moreover, since $x_r\not\succeq x_i$, also $y\not\succeq x_i$ and by the minimality of $r$ we conclude that $q^{x_i}_{y}=0$. It follows that 
\[
\left(\prod_{E\in\mathcal{G}^{\partial x_r}_1}\alpha_E\right)\ \mr{divides} \  q^{x_i}_{x_r}.
\]
By induction, we know that we have already found a KT-class $p^{x_r}$ and we can substitute $q^{x_i}$ by 
\[
q^{x_i}-\left(\frac{ q^{x_i}_{x_r}}{\prod_{E\in\mathcal{G}^{\partial x_r}_1}\alpha_E}\right)p^{x_r}. 
\]
This element verifies (KT1) and (KT2). Property (KT3) holds for every entry with index $\leq x_r$, so that we can proceed recursively (since $X^T$ has a finite number of elements, this procedure will certainly end).

Note now that the matrix whose $(i,j)$-entry is the coefficient of  $q^{x_j}$ in the expansion of $p^{x_i}$ in the basis $\{q^{x}\}_{x\in X^T}$ is invertible (being upper triangular with 1s on the diagonal). This implies that also $\{p^{x}\}_{x\in X^T}$ is a basis.
\end{proof}
\subsection{Palais-Smale Orientation of the Moment Graph}\label{subsec:P-S-moment graph}
\begin{defi}\label{def:P-S-orientation}
A GKM-variety $(X,T)$ is \f{Palais-Smale} if there exists a generic cocharacter $\chi$ of $T$ such that $\mathcal{G}=\mathcal{G}(X,T,\chi)$ satisfies:
\begin{equation}\label{eqn:PalaisSmale}
 \#\mathcal{G}_1^{\partial x}>\#\mathcal{G}_1^{\partial y} \quad \fa x\to y\in \mathcal{G}_1.
\end{equation}
We say that $\mathcal{G}$ is Palais-Smale (PS) oriented if \eqref{eqn:PalaisSmale} holds.
\end{defi}
\begin{rem}
Note that if $\mathcal{G}(X,T,\chi)$ is Palais-Smale, then $\succeq_\chi$ is a partial order.
\end{rem}

\begin{ex}\label{ex:cohomology-generators-loop-quiver}
Consider the quiver Grassmannian $\mr{Gr}_2(M)$ for the loop quiver $\Delta_1$, with $M= A_2 \oplus A_1 \oplus A_1$ where $A_N \cong \C[t]/(t^N)$ as in \cite[Example~7.6]{LaPu2020}. On this quiver Grassmannian the torus $T= (\C^*)^{3+1}$ acts as
\begin{center}
\begin{tikzpicture}[scale=0.5]
\node at (0,3) {$\gamma_1$};
\node at (0,2) {$\gamma_0\gamma_1$};
\node at (0,1) {$\gamma_2$};
\node at (0,0) {$\gamma_3$};

 \def\centerarc[#1](#2)(#3:#4:#5);%
    {
    \draw[#1]([shift=(#3:#5)]#2) arc (#3:#4:#5);
    }
\centerarc[arrows={-angle 90}](0.5,2.5)(76:-76:0.5cm); 
\end{tikzpicture}
\end{center}
such that it becomes a BB-filterable GKM-variety, by \cite[Lemma~7.5]{LaPu2020}. For the cocharacter $\chi : \C^* \to T$ with $\chi(\lambda)= (\lambda,\lambda,\lambda^3,\lambda^4)=(\gamma_0,\gamma_1,\gamma_2,\gamma_3) \in T$, we can also apply Theorem~\ref{trm:comb-moment-graph} to compute its moment graph. The basis of $M$ induced by $\chi$ (s.t. $Q(M,B)$ is attractive) is $B = \{v_1= b_{1,0}, v_2=b_{1,1}, v_3=b_{2,0}, v_4=b_{3,0}\}$. There are four $T$-fixed points and their corresponding subquivers in $Q(M,B)$ are of the form: 
\begin{center}
\begin{tikzpicture}[scale=0.4]
\node at (-1.1,1.5) {$L_1 = $};
\draw[fill=white] (0,3) circle (.12);
\draw[fill=white] (0,2) circle (.12);
\draw[fill=black] (0,1) circle (.12);
\draw[fill=black] (0,0) circle (.12);

 \def\centerarc[#1](#2)(#3:#4:#5);%
    {
    \draw[#1]([shift=(#3:#5)]#2) arc (#3:#4:#5);
    }
\centerarc[arrows={-angle 90}](0,2.5)(76:-76:0.5cm); 
\end{tikzpicture}$\quad \ $
\begin{tikzpicture}[scale=0.4]
\node at (-1.1,1.5) {$L_2 = $};
\draw[fill=white] (0,3) circle (.12);
\draw[fill=black] (0,2) circle (.12);
\draw[fill=white] (0,1) circle (.12);
\draw[fill=black] (0,0) circle (.12);

 \def\centerarc[#1](#2)(#3:#4:#5);%
    {
    \draw[#1]([shift=(#3:#5)]#2) arc (#3:#4:#5);
    }
\centerarc[arrows={-angle 90}](0,2.5)(76:-76:0.5cm); 
\end{tikzpicture}$\quad \ $
\begin{tikzpicture}[scale=0.4]
\node at (-1.1,1.5) {$L_3 = $};
\draw[fill=white] (0,3) circle (.12);
\draw[fill=black] (0,2) circle (.12);
\draw[fill=black] (0,1) circle (.12);
\draw[fill=white] (0,0) circle (.12);

 \def\centerarc[#1](#2)(#3:#4:#5);%
    {
    \draw[#1]([shift=(#3:#5)]#2) arc (#3:#4:#5);
    }
\centerarc[arrows={-angle 90}](0,2.5)(76:-76:0.5cm); 
\end{tikzpicture}$\quad \ $
\begin{tikzpicture}[scale=0.4]
\node at (-1.1,1.5) {$L_4 = $};
\draw[fill=black] (0,3) circle (.12);
\draw[fill=black] (0,2) circle (.12);
\draw[fill=white] (0,1) circle (.12);
\draw[fill=white] (0,0) circle (.12);

 \def\centerarc[#1](#2)(#3:#4:#5);%
    {
    \draw[#1]([shift=(#3:#5)]#2) arc (#3:#4:#5);
    }
\centerarc[arrows={-angle 90}](0,2.5)(76:-76:0.5cm); 
\end{tikzpicture}
\end{center}
The induced basis of the cocharacter $\chi' : \C^* \to T$ with $\chi'(\lambda)= (\lambda^3,\lambda,\lambda^2,\lambda^3)=(\gamma_0,\gamma_1,\gamma_2,\gamma_3) \in T$, is $B'=\{v_1= b_{1,0}, v_2=b_{2,0}, v_3=b_{3,0}, v_4=b_{1,1}\}$ and the coefficient quivers of the fixed points in this basis are:
\begin{center}
\begin{tikzpicture}[scale=0.4]
\node at (-1.1,1.5) {$L_1' = $};
\draw[fill=white] (0,3) circle (.12);
\draw[fill=white] (0,2) circle (.12);
\draw[fill=black] (0,1) circle (.12);
\draw[fill=black] (0,0) circle (.12);

 \def\centerarc[#1](#2)(#3:#4:#5);%
    {
    \draw[#1]([shift=(#3:#5)]#2) arc (#3:#4:#5);
    }
\centerarc[arrows={-angle 90}](-1.5,1.5)(42:-42:2.15cm); 
\end{tikzpicture}$\quad \ $
\begin{tikzpicture}[scale=0.4]
\node at (-1.1,1.5) {$L_2' = $};
\draw[fill=white] (0,3) circle (.12);
\draw[fill=black] (0,2) circle (.12);
\draw[fill=white] (0,1) circle (.12);
\draw[fill=black] (0,0) circle (.12);

 \def\centerarc[#1](#2)(#3:#4:#5);%
    {
    \draw[#1]([shift=(#3:#5)]#2) arc (#3:#4:#5);
    }
\centerarc[arrows={-angle 90}](-1.5,1.5)(42:-42:2.15cm); 
\end{tikzpicture}$\quad \ $
\begin{tikzpicture}[scale=0.4]
\node at (-1.1,1.5) {$L_3' = $};
\draw[fill=white] (0,3) circle (.12);
\draw[fill=black] (0,2) circle (.12);
\draw[fill=black] (0,1) circle (.12);
\draw[fill=white] (0,0) circle (.12);

 \def\centerarc[#1](#2)(#3:#4:#5);%
    {
    \draw[#1]([shift=(#3:#5)]#2) arc (#3:#4:#5);
    }
\centerarc[arrows={-angle 90}](-1.5,1.5)(42:-42:2.15cm); 
\end{tikzpicture}$\quad \ $
\begin{tikzpicture}[scale=0.4]
\node at (-1.1,1.5) {$L_4' = $};
\draw[fill=black] (0,3) circle (.12);
\draw[fill=white] (0,2) circle (.12);
\draw[fill=white] (0,1) circle (.12);
\draw[fill=black] (0,0) circle (.12);

 \def\centerarc[#1](#2)(#3:#4:#5);%
    {
    \draw[#1]([shift=(#3:#5)]#2) arc (#3:#4:#5);
    }
\centerarc[arrows={-angle 90}](-1.5,1.5)(42:-42:2.15cm); 
\end{tikzpicture}
\end{center}
For $\chi$ we obtain the moment graph on the left below, and for $\chi'$ we get the moment graph on the right hand side:
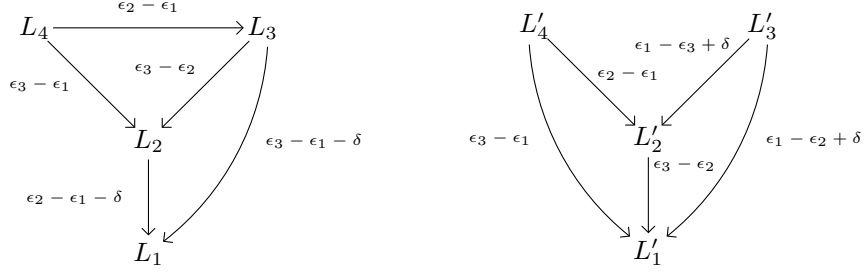
\begin{center}
\begin{tikzpicture}[scale=0.75]

\node at (-2,0) {$L_4$};
\node at (2,0) {$L_3$};
\node at (0,-2) {$L_2$};
\node at (0,-4) {$L_1$};

\node at (0,0.4) {\tiny $\epsilon_2-\epsilon_1$}; 
\node at (-1.9,-1) {\tiny $\epsilon_3-\epsilon_1$}; 
\node at (0.3,-0.7) {\tiny $\epsilon_3-\epsilon_2$}; 
\node at (2.9,-2) {\tiny $\epsilon_3-\epsilon_1-\delta$}; 
\node at (-1.3,-3) {\tiny $\epsilon_2-\epsilon_1-\delta$}; 

\draw[arrows={-angle 90}, shorten >=7, shorten <=7]  (-2,0) -- (2,0); 
\draw[arrows={-angle 90}, shorten >=7, shorten <=7]  (2,0) -- (0,-2); 
\draw[arrows={-angle 90}, shorten >=7, shorten <=7]  (-2,0) -- (0,-2); 
\draw[arrows={-angle 90}, shorten >=7, shorten <=7]  (0,-2) -- (0,-4); 

 \def\centerarc[#1](#2)(#3:#4:#5);%
    {
    \draw[#1]([shift=(#3:#5)]#2) arc (#3:#4:#5);
    }
\centerarc[arrows={-angle 90}](-2.7,0)(-4:-52:4.8cm); 
\end{tikzpicture}$\quad \quad \quad$
\begin{tikzpicture}[scale=0.75]

\node at (-2,0) {$L_4'$};
\node at (2,0) {$L_3'$};
\node at (0,-2) {$L_2'$};
\node at (0,-4) {$L_1'$};

\node at (-0.35,-0.85) {\tiny $\epsilon_2-\epsilon_1$}; 
\node at (0.6,-0.35) {\tiny $\epsilon_1-\epsilon_3+\delta$}; 
\node at (0.63,-2.46) {\tiny $\epsilon_3-\epsilon_2$}; 
\node at (2.9,-2) {\tiny $\epsilon_1-\epsilon_2+\delta$}; 
\node at (-2.6,-2) {\tiny $\epsilon_3-\epsilon_1$}; 

\draw[arrows={-angle 90}, shorten >=7, shorten <=7]  (2,0) -- (0,-2); 
\draw[arrows={-angle 90}, shorten >=7, shorten <=7]  (-2,0) -- (0,-2); 
\draw[arrows={-angle 90}, shorten >=7, shorten <=7]  (0,-2) -- (0,-4); 

 \def\centerarc[#1](#2)(#3:#4:#5);%
    {
    \draw[#1]([shift=(#3:#5)]#2) arc (#3:#4:#5);
    }
\centerarc[arrows={-angle 90}](-2.7,0)(-4:-51:4.8cm); 
\centerarc[arrows={-angle 90}](2.7,0)(180+4:180+51:4.8cm); 
\end{tikzpicture}
\end{center}
Observe that the moment graph on the right hand side is PS-oriented, while the one on the left hand side is not.
\end{ex}
\begin{rem}
Example~\ref{ex:cohomology-generators-loop-quiver} shows that the choice of an attractive grading of a nilpotent $\Delta_n$-representation $M$, corresponding to the choice of a generic cocharacter (see Proposition~\ref{prop:generic-cochar}), determines an orientation the moment graph $\mathcal{G}(X,T)$ for the $T$-action on $\mr{Gr}_\mb{e}(M)$ as in Definition~\ref{def:moment-graph}.	For our computations we assume that the coefficient quiver of $M$ is attractively aligned (see Remark~\ref{rem:alignability}). Hence the choice of another attractive grading corresponds to a permutation of the basis vectors preserving each set $B^{(i)}$. We will deal with a particular class of these permutations in Section~\ref{sec:Permutation-Action}.
\end{rem}
The following lemma is proven by Tymoczko in \cite[Lemma~2.16]{Tymoczko2008} and provides uniqueness of KT-classes with respect to a fixed PS-orientation.
\begin{lma}\label{lma:K-T-classes-unique}
Let $(X,T)$ be a GKM-variety and let $\chi\in\mathfrak{X}^*(T)$ be a genereic cocharacter such that $\mathcal{G}(X,T,\chi)$ is Palais-Smale. If $p^x = (p^x_y)_{y \in X^T}$, $q^x = (q^x_y)_{y \in X^T}$ are two Knutson-Tao classes (with respect to $\succeq_\chi$) corresponding to $x \in X^T$, then $p^x_y = q^x_y$ for each $y \in X^T$.
\end{lma}
\begin{ex}\label{ex:K-T-basis-loop-quiver}
In \cite[Example~7.6]{LaPu2020}, we use \cite[Theorem~2.10]{LaPu2020} to compute an $H_T^\bullet(\textrm{pt})$-module basis of $H_T^\bullet(X)$ for the quiver Grassmannian from Example~\ref{ex:cohomology-generators-loop-quiver}. Now we can apply Proposition~\ref{prop:K-T-basis} to $\mathcal{G}(X,T,\chi')$ since this moment graph is acyclic. Thus, the KT-classes below form an alternative basis of $H_T^\bullet(X)$.
\begin{align*}
p^1 &= (1,1,1,1)\\
p^2 &= (0,\epsilon_3-\epsilon_2,\epsilon_2-\epsilon_1-\delta,\epsilon_3-\epsilon_1)\\
p^3 &= (0,0,(\epsilon_3-\epsilon_1-\delta)(\epsilon_2-\epsilon_1-\delta),0)\\
p^4 &= (0,0,0,(\epsilon_3-\epsilon_1)(\epsilon_2-\epsilon_1))
\end{align*}
Observe that determining a basis consisting of KT-classes is in general significantly easier than computing the basis from \cite[Theorem~2.10]{LaPu2020}. In fact, the construction in \cite{LaPu2020} required the computation of Euler classes. This is in general a non trivial task, which may be solved by constructing $T$-equivariant desingularisations, as explained in \cite[Appendix~A]{LaPu2020}. 

Moreover, since $\mathcal{G}(X,T,\chi')$ is PS-oriented, we can apply Lemma~\ref{lma:K-T-classes-unique} to conclude that the above basis is the unique basis with the properties in Definition~\ref{def:K-T-class}. 
\end{ex}
\subsection{Homogeneous Representations}\label{subsec:Homogeneous-Forests}
In this section, we introduce a class of quiver Grassmannians which are Palais-Smale.
\begin{defi}\label{def:homog-forest}
A nilpotent representation $M \in \mr{rep}_\C(\Delta_n)$ is called \f{homogeneous} if there exists a basis $B$ such that:
\begin{enumerate}
\item[(H0)] $M$ is attractive,
\item[(H1)] $Q(M,B)$ is aligned,
\item[(H2)] $Q(M,B)$ is gapless, i.e.: for each $(a:i\to i+1) \in \Z_n$ and $k \in [m_i]$ with $M_a v^{(i)}_k = v^{(i+1)}_{k'}$, and $m_{i+1} - k' > 0$ it holds that
\[ M_a v^{(i)}_{k+r} = v^{(i+1)}_{k'+r} \quad \fa r \in [m_{i+1}-k'].\]
\end{enumerate}
\end{defi}
\begin{ex}\label{ex:homog-cyclic-rep}
Let 
\[ M = \bigoplus_{i \in \Z_n} U_i(N) \otimes \C^{k_i} \]
be a $\Delta_n$-representation, with $N \in \Z_{\geq 1}$ and $k_i \in \Z_{\geq 0}$ for all $i \in \Z_n$. Independently of the choice of $N$ and the $k_i$'s, the representation $M$ is homogeneous of we consider the basis as constructed in the proof of \cite[Proposition~4.8]{LaPu2020}. The quiver Grassmannians for $M$ of the above form are of particular interest, because they can be used to approximate the (degenerate) affine Grassmannian (see \cite[Proposition~7.4]{LaPu2020}) and the (degenerate) affine flag variety (see \cite[Theorem~3.7]{Pue2020}), both of type $\mathfrak{gl}_n$. 
\end{ex}
\begin{prop}\label{prop:homg-forest-are-P-S}
Let $M\in \mr{rep}_\C(\Delta_n)$ be  homogeneous nilpotent. If $T$ acts on $X:=\mr{Gr}_\mb{e}(M)$ as in Proposition~\ref{prop:generic-cochar} and $\chi$ is the corresponding cocharacter of the attractive grading of $M$, then $\mathcal{G}(X,T,\chi)$ is Palais-Smale. 
\end{prop}
\begin{proof} 
The structure of the moment graph $\mathcal{G}(X,T,\chi)$ is described in Theorem~\ref{trm:comb-moment-graph}. Hence, the number of arrows starting at a point in the moment graph equals the number of fundamental mutations starting at the corresponding fixed point, and the dimension of the associated cell in $X$. It remains to show that fundamental mutations of the fixed points are dimension decreasing for the corresponding cells if $M$ is homogeneous.

Let $\mu:L \to L'$ be a fundamental mutation and denote the successor closed subsegments (belonging to $L$ and $L'$) in the $d$-many trees of $Q(M,B)$ by $S_1,\dots,S_d$ and $S_1',\dots,S_d'$. We can define a function \(h: \{S_j,S_j' \vert j \in [d]\} \to \Z_{\geq 0}\), counting the mutation relations starting at the respective subsegment of $L$ and $L'$. By Theorem~\ref{trm:comb-moment-graph}, we obtain that 
\[ 
\dim C_L = \sum_{j \in [d]} h(S_j)
\]
and analogously for $L'$. Let $j_s$ and $j_t$ be the indices of the mutation $\mu$ as introduced in Remark~\ref{rem:T-weights-one-dim-orbits}. Hence $S_j'  = S_j$ for all $j \in [d] \setminus \{ j_s,j_t\}$.

Since $M$ is homogeneous and $j_t > j_s$, it follows that the coefficient quiver of $S_{j_t}'$ is a subquiver of the one for $S_{j_s}$. Hence, there exists a $k > 0$ such that $h(S_{j_s}) = h(S_{j_t}')+k$, because $M$ is homogeneous and $j_t > j_s$. ${S}_{j_s}'$ is remaining part of ${S}_{j_s}$ after the mutation $\mu$. We compute $h({S}_{j_s}') < h({S}_{j_t})+k$, since by construction there are no possible relations from ${S}_{j_s}'$ with $S_{j_t}'$. 

For all $j \in [d] \setminus \{ j_s,j_t\}$ we obtain $h(S_j') \leq h(S_j)$, since $M$ is homogeneous and the components remain unchanged. This implies that $\dim C_{L'} < \dim C_L$ holds for all fixed points which are connected by an edge $L \to L'$ in the moment graph. Hence all fundamental mutations are strictly dimension decreasing and the corresponding moment graph satisfies the Palais-Smale property.
\end{proof}
\begin{rem}
If we reorder the points in the coefficient quiver of $M$ from Example~\ref{ex:cohomology-generators-loop-quiver} increasingly by their $\C^*$-weight for the action by $\chi'$, then the corresponding coefficient quiver $Q(M,B')$ is homogeneous. Note that for the basis $B$ induced by $\chi$, the coefficient quiver  $Q(M,B)$ violates property (H2).
\end{rem}
\begin{trm}\label{trm:homg-forest-have-unique-KT-basis}
Let $M\in \mr{rep}_\C(\Delta_n)$ a homogeneous nilpotent representation. If $T$ acts on $X:=\mr{Gr}_\mb{e}(M)$   as in Lemma~\ref{lma:T-action-extends-to-quiver-Grass} and the generic cocharacter $\chi$ is as in Proposition~\ref{prop:homg-forest-are-P-S}, then $H_T^\bullet(X)$ has a unique KT-basis (w.r.t. $\succeq_\chi$).
\end{trm}
\begin{proof}
The existence of a KT-basis follows from Theorem~\ref{thm:ExistenceKTBasis} because $(X,T)$ is a BB-filterable GKM-variety by Theorem~\ref{trm:attr-forests-are-GKM}. Now, Lemma~\ref{lma:K-T-classes-unique} implies the uniqueness of this KT-basis since $\mathcal{G}(X,T,\chi)$ is Palais-Smale by Proposition~\ref{prop:homg-forest-are-P-S}. 
\end{proof}
\begin{rem}From now on, whenever we consider a homogeneous nilpotent $\Delta_n$-representation $M$, and a quiver Grassmannian $X=\mr{Gr}_\mb{e}(M)$, we implicitely assume that it is equipped with a $T$-action as in  Lemma~\ref{lma:T-action-extends-to-quiver-Grass} and that we fixed a cocharacter $\chi$ such that the corresponding moment graph $\mathcal{G}(X,T,\chi)$ is Palais-Smale.
\end{rem}
\section{Permutation Action}\label{sec:Permutation-Action}
In this section we study the action of the group of permutations inside the group of quiver representation automorphisms which exchanges the isomorphic indecomposable summands of the quiver representation. The action on the quiver representation induces an action on the corresponding quiver Grassmannians and the moment graphs from Section~\ref{subsec:Moment-Graph}. Moreover it induces a geometric action on the equivariant cohomology as studied in Section~\ref{subsec:Geo-action-cohomology}.

Assume that the nilpotent representation $M \in \mr{rep}_\C(\Delta_n)$ is isomorphic to a direct sum of $d$-many indecomposable summands $U_1, \ldots, U_d \in \mr{rep}_\C(Q)$. If $d_0\leq d$ is the number of distinct isomorphism classes of the $U_i$'s, we can renumber them in such a way that there exist $k_1, \ldots, k_{d_0}\in\Z_{\geq 1}$ such that $d=k_1+\ldots+k_{d_0}$ and 
\[
U_i\cong U_{i'}\quad \Leftrightarrow\quad \exists j\in[d_0] \hbox{ such that }i,i'\in[k_1+\ldots+k_{j-1}+1, k_1+\ldots+k_j],
\]
where by convention we set $k_0=0$. For $j\in[d_0]$, we denote by $k_j':=k_1+\ldots+k_j$, and by
$\underline{k'}$ the set $\{k_1', k_2', \ldots, k_{d_0}'=d\}$. Thus, we have
\begin{equation}\label{eqn:DirectSumDecp}
M \cong \bigoplus_{j \in [d_0]} U_{k_j'} \otimes \C^{k_j}.
\end{equation}

 Notice that any multiplicity space $\C^{k_j}$ is equipped with an action of the symmetric group $\mathfrak{S}_{k_j}$ which permutes the coordinates and we obtain in this way an action of $\mathfrak{S}_{k_1}\times\ldots\times \mathfrak{S}_{k_{d_0}}$ on $M$ via quiver representation automorphisms. In what follows, we will investigate the induced action on quiver Grassmannians and on their cohomology rings. 
 
It will be convenient to realise the group $\mathfrak{S}_{k_1}\times\ldots\times \mathfrak{S}_{k_{d_0}}$ as a subgroup of $\mathfrak{S}_d$: let $\mathfrak{S}_{\underline{k}}\subseteq \mathfrak{S}_d$ be the subgroup of the symmetric group over $d$ letters which stabilises the following subsets of $[d]$: 
\[
[k'_1], \ [k'_1+1,k_2'] \ldots,\ [k_{d_0-1}'+1,d]. 
\]
 \begin{rem}
 Notice that a generic point in the representation variety for $\Delta_n$ will have indecomposable summands pairwise not isomorphic, and hence the group $\mathfrak{S}_{\underline{k}}$ will be trivial. The representations which have interesting permutation group actions are hence special, but still a very big family, considering that they comprise the whole family of $n$-step flag varieties in $\C^m$ (corresponding to Quiver Grassmannians for the representation $U_1(n)\otimes\C^m$).
 \end{rem}

Consider now the vector space obtained by taking the direct sum of the multiplicity spaces $\C^{k_1}\oplus \C^{k_2}\oplus\ldots\oplus\C^{k_{d_0}}\cong \C^d$, and let $(e_1, \ldots, e_{k_1},e_{k_1+1}, \ldots e_{k_1+k_2}, \ldots, e_d)$ be the standard basis of $\C^{d}$, ordered in such a way that $(e_{k_{j-1}'+1}, \ldots, e_{k_j'})$ is the standard basis of $\C^{k_j}$ (the multiplicity space of $U_j$). Observe that in this way we are identifying $U_i$ with $U_{k_j'}\otimes e_{i}$ if $i\in[k_{j-1}'+1,k_j']$.

Let $Q(M,B)$ be the coefficient quiver of $M$, where $B$ has been chosen in such a way that  the connected components of $Q(M,B)$ are in bijection with the indecomposable nilpotents $U_1, \ldots, U_d$.  We also write $U_i$ for  $Q(U_i,B) \subset Q(M,B)$ and (under the isomorphism \eqref{eqn:DirectSumDecp}) we obtain 
\[ B\cap U_i=\Big\{b_t\otimes e_i  \ \Big\vert
\begin{array}{c}
\hbox{if }i\in[k_{j-1}'+1,k_j']\\
 \hbox{and }(b_t)_{t=1,\ldots, \ell_j}\ \hbox{ is a basis of }U_{k_j'}
\end{array} \Big\}.  \] 
Then $\mathfrak{S}_{\underline{k}}$ acts on $ M \cong \bigoplus_{j \in [d_0]}  U_{k_j'} \otimes \C^{k_j}$ via
\[
\sigma(u\otimes e_i)=u\otimes e_{\sigma^{-1}(i)}, \qquad u\in B\cap U_{k_j'}, \ i\in[k_{j-1}'+1, k_j'].
\]
Thus, any $\sigma\in\mathfrak{S}_{\underline{k}}$ induces an oriented graph automorphism of $Q(M,B)$ having the property that $\sigma(U_{i})=U_{\sigma^{-1}(i)}$.
\begin{ex}\label{ex:Permutation-Action-Loop-Quiver}
Let $M$ be as in Example~\ref{ex:homog-cyclic-rep}, then every quiver Grassmannian $\mr{Gr}_\mb{e}(M)$ admits an action by $\prod_{i \in \Z_n}  \mathfrak{S}_{k_i}$.
\end{ex}
\begin{ex}\label{ex:PermActionFlagVar1}
Let us consider the $\Delta_n$-representation $\bigoplus_{j=1}^{m} U_1(n)\simeq U_1(n) \otimes \C^m$. 
By the previous discussion we get an action of  $\mathfrak{S}_m$ on it. In particular, this is a special case of the previous example.
\end{ex}
Let $T=(\C^*)^1\times (\C^*)^d$ be an algebraic torus of dimension $d+1$. Then $\mathfrak{S}_{\underline{k}}$ acts on it as follows: for any $\sigma\in\mathfrak{S}_{\underline{k}}$ and for any $(\gamma_0,\gamma_1, \ldots, \gamma_d)\in T$ we set
\[
\sigma \cdot (\gamma_0,\gamma_1, \ldots, \gamma_d, ):=(\gamma_0,\gamma_{\sigma(1)}, \ldots, \gamma_{\sigma(d)}).
\]
This also induces an $\mathfrak{S}_{\underline{k}}$-action on $\mathfrak{X}^*(T)$ by setting $\sigma(\alpha)(t):=\alpha(\sigma^{-1}(t))$ for any $\sigma\in\mathfrak{S}_{\underline{k}}$, $\alpha\in\mathfrak{X}^*(T)$ and $t\in T$.

If $\sigma\in\mathfrak{S}_{\underline{k}}$ and $t\in T$, we will very often denote the element $\sigma  \cdot t$ by $t^\sigma$. We can now consider the semi-direct product $\mathfrak{S}_{\underline{k}}\ltimes T$ with commutation relation given by
\[
\sigma t=t^\sigma \sigma\qquad \forall\sigma\in\mathfrak{S}_{\underline{k}}, \ t\in T.
\]
\begin{lem}\label{lem:geomSkAction}Let $M \in \mr{rep}_\C(\Delta_n)$ be nilpotent. The group $\mathfrak{S}_{\underline{k}}\ltimes T$ acts on $\mr{Gr}_\mb{e}(M)$ via
\[(\sigma t) N=\sigma(t. N),\]
where $.$ denotes the torus action from Lemma \ref{lma:T-action-extends-to-quiver-Grass}.
\end{lem}
\begin{proof}
First of all, we observe that $N\in\ \mr{Gr}_\mb{e}(M)$ implies $\sigma(N)\in \mr{Gr}_\mb{e}(M)$ for any $\sigma\in \mathfrak{S}_{\underline{k}}$: $\sigma$ is by definition an automorphism of the $\Delta_n$-representation $M$ and hence $\sigma=(\sigma^{(i)})_{i\in \Z_n}$, where $\sigma^{(i)}:M^{(i)}\rightarrow M^{(i)}$ is a vector space isomorphism for any $i\in \Z_n$ and $\sigma^{(i+1)} \circ M_a=M_a\circ \sigma^{(i)}$ for any edge $a:i\to i+1$, therefore $\dim_\C(\sigma^{(i)}N^{(i)})=\dim_\C N^{(i)}$ and 
\[
M_a((\sigma N)^{(i)})=M_a(\sigma^{(i)}N^{(i)})=\sigma^{(i+1)}(M_aN^{(i)})\subseteq\sigma^{(i+1)}(N^{(i)})=(\sigma N)^{(i+1)}.\]
To conclude it is enough to show that for any $\sigma\in\mathfrak{S}_{\underline{k}}$ and $t=(\gamma_0,\gamma_1, \ldots,\gamma_d)\in T$ it holds $\sigma(t.( b\otimes e_i))=t^\sigma . (\sigma (b\otimes e_i))$ for all $b\in U_{k_j'} \cap B$ and $i\in[k_{j-1}'+1, k_j']$. Notice that for $b\in U_{k_j'} \cap B$ there exists a unique $p$ such that, in the notation of \S\ref{eqn:indec-basis}, $b\otimes e_i$ for $b=b_{k_j',p}$ gets identified with $b_{i,p}$ under the isomorphism \eqref{eqn:DirectSumDecp}. We hence have
\begin{align*}
\sigma(t.(b\otimes e_i))&=\sigma( \gamma_0^p\gamma_i b\otimes e_{i})\\
&=\gamma_0^p\gamma_i  b\otimes e_{\sigma^{-1}(i)}\\
&=\gamma_0^p\gamma_i  b\otimes e_{\sigma^{-1}(i)}\\
&=t^\sigma .( b\otimes e_{\sigma^{-1}(i)})\\
&=t^\sigma. \sigma(b\otimes e_i).
\end{align*}
\end{proof}
The action in the above lemma restricts to an  $\mathfrak{S}_{\underline{k}}$-action on $\mr{Gr}_\mb{e}(M)$. 
We denote by $\overline{\mathcal{G}(X,T)}$ the underlying unoriented graph of the moment graph $\mathcal{G}(X,T, \chi)$.
\begin{cor}Let $M \in \mr{rep}_\C(\Delta_n)$ be nilpotent.
The $\mathfrak{S}_{\underline{k}}$-action on $X=\mr{Gr}_\mb{e}(M)$ induces an automorphism of the graph $\overline{\mathcal{G}(X,T)}$.
\end{cor}
\begin{proof}By Lemma~\ref{lem:geomSkAction}, the $\mathfrak{S}_{\underline{k}}$-action normalizes the torus action. From the computations in the proof of this lemma, we see that
\begin{equation}\label{eqn:SkActionGraph}
L\stackrel{\alpha}{-\!\!\!-\!\!\!-\!\!-} L'\in \overline{\mathcal{G}(X,T)}_1\quad\Leftrightarrow\quad \sigma(L)\stackrel{\sigma(\alpha)}{-\!\!\!-\!\!\!-\!\!-} \sigma(L')\in \overline{\mathcal{G}(X,T)}_1 \ (\forall\ \sigma\in\mathfrak{S}_{\underline{k}}).
\end{equation}
Hence every $\sigma$ sends torus fixed points to torus fixed points and one-dimensional torus orbits to one-dimensional torus orbit.
\end{proof}
\begin{rem}In \cite{Kaji2015}, Kaji introduced the notion of moment graphs admitting Coxeter group symmetries (see \cite[Definition 4.1]{Kaji2015}). He assumed the regularity of the graph, but the same definition makes sense without the regularity assumption. Thus, the above corollary tells us that the moment graph $\overline{\mathcal{G}(X,T)}$  admits Coxeter symmetries. 
\end{rem}
\begin{ex}\label{ex:S2ActionLoopA2A1A1}Let us consider the graph $\mathcal{G}(X,T,\chi')$ from Example \ref{ex:cohomology-generators-loop-quiver}. Since $M=A_2\oplus A_1\oplus A_1$ we have an $\mathfrak{S}_2$-action on $X=\mr{Gr}_{2}(M)$ induced by the exchange of the two copies of $A_1$. The corresponding automorphism of $\overline{\mathcal{G}(X,T)}$ is the automorphism which exchanges  the two central vertices $L_1'$ and $L_2'$, and fixes the other two.
\end{ex}
\subsection{The Permutation Action on Nilpotent Representations of the Cycle}
In this subsection, we collect some properties of the $\mathfrak{S}_{\underline{k}}$-action.
\begin{lem}\label{lem:reflectionEdge}Let $M \in \mr{rep}_\C(\Delta_n)$ be homogeneous and nilpotent, $X=\mr{Gr}_\mb{e}(M)$ and  $\mathcal{G}=\mathcal{G}(X,T,\chi)$. Let $h, l\in[k_{j-1}'+1,k_j']$ with $h<l$ (for some $j\in[d_0]$). Let $L\in X^T$ be such that  $\#\mathcal{G}_1^{\partial\sigma_{h,l}(L)}<\#\mathcal{G}_1^{\partial L}$, where $\sigma_{h,l}:=(h,l)\in\mathfrak{S}_{\underline{k}}$ is the transposition which exchanges $h$ and $l$. Then $\sigma_{h,l}(L)$ is a fundamental mutation of $L$, and $T$ acts on the corresponding one-dimensional orbit via the character $\epsilon_l-\epsilon_h$.
\end{lem}
\begin{proof}
We start by showing that $\sigma_{h,l}(L)$ and $L$ are in mutation relation: By assumption $U_h$ and $U_l$ are isomorphic summands of a nilpotent $M\in \mr{rep}_\C(\Delta_n)$. Hence their segments in $Q(M,B)$ are isomorphic equioriented strings of the same length which both end over the same $i \in \Z_n$. For $j \in [d]$ let $S_j$ and $S_j'$ denote the segments of $L$ and $\sigma_{h,l}(L)$ in $Q(U_j,B)$. Since $\sigma_{h,l}$ exchanges the subsegments in $Q(U_h,B)$ and $Q(U_h,B)$ we obtain $S_j= S_j'$ for all $j \in [d]\setminus \{ h,l\}$, and $\ell(S_h)= \ell(S_l')$ and $\ell(S_l)= \ell(S_h')$, where $\ell(S_j)$ denotes the length of segment $S_j$. This implies that $L$ and $\sigma_{h,l}(L)$ are in mutation relation via the movement of an equiorented string of length $\vert \ell(S_h) - \ell(S_l) \vert$.

Since $M$ is homogeneous, $G$ is Palais-Smale by Propostion~\ref{prop:homg-forest-are-P-S}. Thus the relation has to be oriented from $L$ to $\sigma_{h,l}(L)$ because $\#\mathcal{G}_1^{\partial\sigma_{h,l}(L)}<\#\mathcal{G}_1^{\partial L}$ holds by assumption. This also implies that $ \ell(S_h) > \ell(S_l)$. Hence we can view $\sigma_{h,l}$ acting on $L$ as moving the predecessor closed subsegment of length $ \ell(S_h)- \ell(S_l)$ from $S_h$ to $S_l$. and the label of the corresponding edge in the moment graph is computed as $\epsilon_l-\epsilon_h$ by Theorem~\ref{trm:comb-moment-graph}.
\end{proof}

\begin{ex}\label{ex:flag-variety} If we take $m=n+1$ in Example~\ref{ex:PermActionFlagVar1} and $\mb{e}=(1,2,\ldots, n)$ then $\mr{Gr}_\mb{e}(M)\simeq \mathcal{F}l_{n+1}$, the variety of complete flags in $\mathbb{C}^{n+1}$. In this case, every edge is induced by some transposition $\sigma_{h,l} \in \mathfrak{S}_{n+1}$. 
\end{ex}
In general, not every edge of $\mathcal{G}(X,T,\chi)$ is induced by a transposition $\sigma_{h,l}\in\mathfrak{S}_{\underline{k}}$.
\begin{ex}Let us consider the graph $\mathcal{G}=\mathcal{G}(X,T,\chi')$ from Example \ref{ex:cohomology-generators-loop-quiver} and the automorphism $\sigma$ of the underlying unoriented graph from Example \ref{ex:S2ActionLoopA2A1A1}. In the notation of the previous lemma, $\sigma=\sigma_{2,3}$. Moreover, $\sigma_{2,3}(L_2')=L_1'$ and $L_2'\stackrel{\epsilon_3-\epsilon_2}{\rightarrow}L_1'\in \mathcal{G}_1$.
\end{ex}
From now on, it will be convenient to write  $\sigma_i$ for $\sigma_{i,i+1}$. Observe that $\sigma_i\in\mathfrak{S}_{\underline{k}}$ if and only if $i\in[d]\setminus \underline{k}'$.
\begin{lem}\label{lem:LengthSimpleRefln}Let $M \in \mr{rep}_\C(\Delta_n)$ be homogeneous and nilpotent. Let $X=\mr{Gr}_\mb{e}(M)$ and  $\mathcal{G}=\mathcal{G}(X,T,\chi)$. Let $i\in[d]\setminus\underline{k'}$ and let $L\in X^T$ be such that $\sigma_i (L)\prec L$. Then $\#\mathcal{G}_1^{\partial \sigma_{i}(L)}=\#\mathcal{G}_1^{\partial L}-1$.
\end{lem}
\begin{proof}
Since $M$ is homogeneous and $\sigma_i (L)\prec L$ we know that $\#\mathcal{G}_1^{\partial \sigma_{i}(L)} \leq\#\mathcal{G}_1^{\partial L}-1$. Equality follows with the same arguments as in the proof of Proposition~\ref{prop:homg-forest-are-P-S}, since the exchanged segments are directly next to each other because $\sigma_i$ is simple. Hence we can compute the difference between the height functions $h$ for $L$ and $\sigma_i (L)$ explicitly.
\end{proof}

\begin{lem}\label{lem:EdgesAdjacentSimpleRefln}Let $M \in \mr{rep}_\C(\Delta_n)$ be homogeneous and nilpotent, let $X=\mr{Gr}_\mb{e}(M)$ and  $\mathcal{G}=\mathcal{G}(X,T,\chi)$. 
Let $L\in X^T$ and $h<l$ with $h,l\in[k_{j-1}'+1,k_j'-1]$ for some $j\in[d_{0}]$ be such that $\#\mathcal{G}_1^{\partial \sigma_{h,l}(L)}=\#\mathcal{G}_1^{\partial L}+1$. Then
\begin{align*}
\{\alpha\mid \sigma_{h,l}(L)\stackrel{\alpha}{\to} L'\in\mathcal{G}_1^{\partial \sigma_{h,l}(L)}\}&=
\{\epsilon_l-\epsilon_{h}\}\cup \{\sigma_{h,l}(\beta)\mid L\stackrel{\beta}{\to} L''\in\mathcal{G}_1^{\partial L}\}\\
&\equiv \{\epsilon_l-\epsilon_{h}\}\cup \{\beta\mid L\stackrel{\beta}{\to} L''\in\mathcal{G}_1^{\partial L}\}\mod \epsilon_l-\epsilon_{h}
\end{align*}
Moreover, if $i\in [k_{j'-1}'+1,k_{j'}-1]$ for some $j'\in[d_0]$ is such that $\sigma_{i}\neq \sigma_{h,l}$, and $\#\mathcal{G}_1^{\partial \sigma_{i}(L)}=\#\mathcal{G}_1^{\partial L}+1$, then $\#\mathcal{G}_1^{\partial \sigma_{i}\sigma_{h,l}(L)}>\#\mathcal{G}_1^{\partial \sigma_{h,l}(L)}$
\end{lem}
\begin{proof}
Let $\mu_\gamma: \sigma_{h,l}(L)\stackrel{\gamma}{\to} L$ be the mutation associated to $\sigma_{h,l}$. Then by assumption about $L$ and $\sigma_{h,l}$, for every $\mu_\beta: L\stackrel{\beta}{\to} L''$ there exists a $\mu_{\tilde{\beta}}: \sigma_{h,l}(L)\stackrel{\tilde{\beta}}{\to} L''$, since $\mu_\beta \circ \mu_\gamma(\sigma_{h,l}(L)) = L''$ and $\sigma_{h,l}$ simply exchanges the role of the segments indexed by $h$ and $l$. This also implies that $\tilde{\beta}$ = $\sigma_{h,l}(\beta)$ and we obtain the first equality since $\#\mathcal{G}_1^{\partial \sigma_{h,l}(L)}=\#\mathcal{G}_1^{\partial L}+1$. The second identity is immediate since $\sigma_{h,l}(\beta) \equiv \beta \mod \epsilon_l-\epsilon_{h}$.

By the choice of $\sigma_i$ it follows that there is a mutation $\mu_\alpha : \sigma_i(L) \stackrel{\alpha}{\to} L$. Accordingly $\sigma_i \circ \sigma_{h,l}(L)$ and $\sigma_{h,l}(L)$ are also in mutation relation. Since both $\sigma_i$ and $\sigma_{h,l}$ raise some segment in the coefficient quiver, it follows that it is oriented as $\mu_{\tilde{\alpha}} : \sigma_i\sigma_{h,l}(L) \stackrel{\alpha}{\to} \sigma_{h,l}(L)$. This implies the claim since $\mathcal{G}$ is Palais-Smale oriented as shown in Proposition~\ref{prop:homg-forest-are-P-S}.
\end{proof}

\begin{rem}For $\mr{Gr}_\mb{e}(M)\simeq \mathcal{F}l_{n+1}$, the above lemma is \cite[Lemma 2.5]{Tymoczko2008}.
\end{rem}

\begin{lem}\label{lem:EdgesAdjacent}Let $M \in \mr{rep}_\C(\Delta_n)$ be homogeneous and nilpotent, let $X=\mr{Gr}_\mb{e}(M)$ and  $\mathcal{G}=\mathcal{G}(X,T,\chi)$. Let $\check{E}: L\stackrel{\alpha_{\check{E}}}{\to} L'\in\mathcal{G}_1$ such that $\#\mathcal{G}^{\partial L'}_1=\# \mathcal{G}^{\partial L}_1-1$. Then
\[\{\alpha \mid L\stackrel{\alpha}{\to} \hat{L} \in\mathcal{G}_1^{\partial L}\}\equiv \{ \alpha_{\check{E}} \} \cup \{\beta\mid L'\stackrel{\beta}{\to} \tilde{L} \in\mathcal{G}_1^{\partial L'}\}\mod \alpha_{\check{E}}.\]
\end{lem}
\begin{proof}
By definition of mutations, there are exactly two segments involved in the mutation $\mu_{\check{E}}$. Hence for each mutation $\mu_E : L' \to \tilde{L}$, which does not involve these segments, there exists a mutation $\mu_{\hat{E}}: L \to \hat{L}$ and $\mu_{\check{E}'}: \hat{L} \to \tilde{L}$ such that $\mu_{\check{E}'} \circ \mu_{\hat{E}} = \mu_{E} \circ \mu_{\check{E}}$. By construction of $\mu_{\hat{E}}$ and the $T$-action as in Lemma~\ref{lma:T-action-extends-to-quiver-Grass}, we obtain $\beta_E = \alpha_{\hat{E}}$.

If a mutation $\mu_E : L' \to \tilde{L}$ involves one of the segments of $\mu_{\check{E}}$, there exists a mutation $\mu_{\overline{E}} : L \to \tilde{L}$ such that $\mu_{\overline{E}} = \mu_E \circ \mu_{\check{E}}$. By construction of the $T$-action this implies that $\beta_E \equiv \alpha_{\overline{E}} \mod \alpha_{\check{E}}$. Now the claim follows since $\#\mathcal{G}^{\partial L'}_1=\# \mathcal{G}^{\partial L}_1-1$.
\end{proof}
\subsection{Geometric Action on Cohomology}\label{subsec:Geo-action-cohomology}
In the proof of the following result we will exploit the Borel construction of equivariant cohomology, which we recall in the case of an algebraic torus $T$ (see \cite[\S1]{Brion2000} for more details). Let us consider the total space $E_T=(\mathbb{C}^\infty\setminus\{0\})^{d+1}$, equipped with the $T$-action
\[(z_1, \ldots, z_d,z_{d+1})t= (z_1t_1, \ldots, z_d t_d,z_{d+1}t_{d+1})\]
for $(z_1, \ldots z_d, z_{d+1})\in E_T$ and $t=(t_1, \ldots, t_d,t_{d+1})\in T$.

Hence, $H_T^\bullet(X)=H^\bullet(E_T\times_T X)$, where $E_T\times_T X$ is the quotient of the product $E_T\times X$ by the equivalence relation $\sim$ given by:
\[
(e,x)\sim (e',x')\quad\Leftrightarrow\quad \exists t\in T\ \colon\  e'=et\hbox{ and }x'=t^{-1}x.
\]
Recall that we denote $H_T^\bullet(\mr{pt})$ by $R$ as in Remark~\ref{rem:torus-characters-part-ii}.
\begin{prop}\label{propn:GeometricPermutation-Action}Let $M \in \mr{rep}_\C(\Delta_n)$ be nilpotent. The $\mathfrak{S}_{\underline{k}}$-action on $X=\mr{Gr}_\mb{e}(M)$ induces an $\mathfrak{S}_{\underline{k}}$-action on $H_T^\bullet(X)$ which under localisation is given by
\[
\sigma\cdot(f_L)_{L\in X^T}=(\sigma(f_{\sigma^{-1}(L)}))_{L\in X^T}, \quad \sigma\in\mathfrak{S}_{\underline{k}},
\]
where the $\mathfrak{S}_{\underline{k}}$-action on $R$ is the one induced by its (linear) action on $\mathfrak{X}^*(T)$.
\end{prop}
\begin{proof}
Let $(f_L)_{L\in X^T}\in H_T^\bullet(X)$, and consider the tuple $(f'_L)_{L\in\ X^T}\in\bigoplus_{L\in X^T}R$ with $f'_L:=\sigma(f_{\sigma^{-1}(L)})$. First of all, we have to verify that
\[L\stackrel{\alpha}{-\!\!\!-\!\!\!-\!\!-} L'\in \overline{\mathcal{G}(X,T)}_1 \quad \Rightarrow\quad f'_{L}-f'_{L'}\in \alpha R.
\]
Let $L\stackrel{\alpha}{-\!\!\!-\!\!\!-\!\!-} L'\in \overline{\mathcal{G}(X,T)}_1$, then by \eqref{eqn:SkActionGraph}, also $\sigma^{-1}(L)\stackrel{\sigma^{-1}(\alpha)}{-\!\!\!-\!\!\!-\!\!-} \sigma^{-1}(L')\in \overline{\mathcal{G}(X,T)}_1$. Since $(f_L)_{L\in X^T}\in H_T^\bullet(X)$, we know that 
\[
f_{\sigma^{-1}(L)}-f_{\sigma^{-1}(L')}\in\sigma^{-1}(\alpha)R.
\]
We deduce that
\[f'_L-f'_{L'}=\sigma(f_{\sigma^{-1}(L)})-\sigma(f_{\sigma^{-1}(L')})=\sigma(f_{\sigma^{-1}(L)}-f_{\sigma^{-1}(L')})\in\sigma(\sigma^{-1}(\alpha))R=\alpha R.\]

Now we want to show that the above algebraic action comes from the geometric action of $\mathfrak{S}_{\underline{k}}$ on $X$. We start by observing that the $\mathfrak{S}_{\underline{k}}$-action on $T$ induces an action on the total space $E_T$: for any 
$(z_1, \ldots, z_d, z_{d+1})\in(\mathbb{C}^\infty\setminus\{0\})^d\times(\mathbb{C}^\infty\setminus\{0\})^1$ and any $\sigma\in\mathfrak{S}_{\underline{k}}$
\[
(z_1, \ldots, z_d, z_{d+1})\sigma=(z_{\sigma(1)}, \ldots, z_{\sigma(d)}, z_{d+1}).
\]
There is hence an $\mathfrak{S}_{\underline{k}}$-action on $E_T\times X$ given by $\sigma\cdot(e,x):=(e \sigma ,\sigma^{-1}x )$, and such an action is constant along $\sim$-equivalence classes: let $e,e'\in E_T$, $x,x'\in X$ and $t\in T$ be such that $e'=et$, $x'=t^{-1}x$ (that is, $(e,x)\sim (e',x')$), then
\[
\sigma\cdot(e',x')=(et\sigma, \sigma^{-1}t^{-1}x)=(e\sigma t^{\sigma^{-1}},(t^{\sigma^{-1}})^{-1}\sigma^{-1}x)\sim (e\sigma, \sigma^{-1}x)=\sigma\cdot(e,x).
\]
We have in this way obtained an $\mathfrak{S}_{\underline{k}}$-action on $E_T\times_T X$, which restricts to an action on $E_T\times_T X^T$, and hence any $\sigma\in\mathfrak{S}_{\underline{k}}$, corresponds to an automorphism of $E_T\times_T X$, resp. $E_T\times_T X^T$, and hence gives us pullback maps 
\[
\sigma^*:H_T^\bullet(X)\rightarrow H_T^\bullet(X), \quad \sigma^*: H^\bullet_T(X^T) \rightarrow H_T^\bullet(X^T)
\]
and $\sigma^*((g_L)_{L\in X^T})=(g'_L)_{L\in X^T}$ with $g'_L=\sigma(g_{\sigma^{-1}(L)})$ for any 
\[(g_L)_{L\in X^T}\in H_T^\bullet(X^T)=\bigoplus_{L\in X^T}R.\]
If $\iota:X^T\hookrightarrow X$ is the inclusion map, then the localization theorem tells us that the following is an injective ring homomorphism: 
\[
\iota^*:H_T^\bullet(X)\hookrightarrow H_T^\bullet(X^T)
\]
To conclude it is sufficient to observe that $\sigma^*\circ \iota^*=\iota^*\circ\sigma^*$ holds by definition of the involved morphisms.
\end{proof}
\begin{ex}In the case of Example~\ref{ex:flag-variety}, this group action is symmetric and was studied for example in \cite{Tymoczko2008}.
\end{ex}
\begin{ex}Let us consider the graph $\mathcal{G}(X,T,\chi')$ from Example~\ref{ex:cohomology-generators-loop-quiver} and the automorphism $\sigma$ of the underlying unoriented graph from Example~\ref{ex:S2ActionLoopA2A1A1}. Consider moreover the KT-class $p_2$ from Example~\ref{ex:K-T-basis-loop-quiver}, then

\begin{center}
\begin{tikzpicture}[scale=0.6]
\node at (-2,0) {$\epsilon_3-\epsilon_1$};
\node at (2,0) {$\epsilon_2-\epsilon_1-\delta$};
\node at (0,-2) {$\epsilon_3-\epsilon_2$};
\node at (0,-4) {$0$};

\node at (-0.35,-0.85) {\tiny $\epsilon_2-\epsilon_1$}; 
\node at (0.6,-0.35) {\tiny $\epsilon_3-\epsilon_1-\delta$}; 
\node at (0.63,-2.46) {\tiny $\epsilon_3-\epsilon_2$}; 
\node at (3.3,-2) {{\tiny $\epsilon_2-\epsilon_1-\delta$}\quad =}; 
\node at (-3.2,-2) {$\sigma_{2,3}\cdot$ \  {\tiny $\epsilon_3-\epsilon_1$}}; 

\draw[arrows={-angle 90}, shorten >=7, shorten <=7]  (2,0) -- (0,-2); 
\draw[arrows={-angle 90}, shorten >=7, shorten <=7]  (-2,0) -- (0,-2); 
\draw[arrows={-angle 90}, shorten >=7, shorten <=7]  (0,-2) -- (0,-4); 

 \def\centerarc[#1](#2)(#3:#4:#5);%
    {
    \draw[#1]([shift=(#3:#5)]#2) arc (#3:#4:#5);
    }
\centerarc[arrows={-angle 90}](-2.7,0)(-4:-51:4.8cm); 
\centerarc[arrows={-angle 90}](2.7,0)(180+4:180+51:4.8cm); 
\end{tikzpicture}\ 
\begin{tikzpicture}[scale=0.6]
\node at (-2,0) {$\epsilon_2-\epsilon_1$};
\node at (2,0) {$\epsilon_3-\epsilon_1-\delta$};
\node at (0,-2) {$0$};
\node at (0,-4) {$\epsilon_2-\epsilon_3$};

\node at (-0.35,-0.85) {\tiny $\epsilon_2-\epsilon_1$}; 
\node at (0.6,-0.35) {\tiny $\epsilon_3-\epsilon_1-\delta$}; 
\node at (0.63,-2.46) {\tiny $\epsilon_3-\epsilon_2$}; 
\node at (2.9,-2) {{\tiny $\epsilon_2-\epsilon_1-\delta$}}; 
\node at (-2.6,-2) {{\tiny $\epsilon_3-\epsilon_1$}}; 

\draw[arrows={-angle 90}, shorten >=7, shorten <=7]  (2,0) -- (0,-2); 
\draw[arrows={-angle 90}, shorten >=7, shorten <=7]  (-2,0) -- (0,-2); 
\draw[arrows={-angle 90}, shorten >=7, shorten <=7]  (0,-2) -- (0,-4); 

 \def\centerarc[#1](#2)(#3:#4:#5);%
    {
    \draw[#1]([shift=(#3:#5)]#2) arc (#3:#4:#5);
    }
\centerarc[arrows={-angle 90}](-2.7,0)(-4:-51:4.8cm); 
\centerarc[arrows={-angle 90}](2.7,0)(180+4:180+51:4.8cm); 
\end{tikzpicture}
\end{center}
Observe that $\sigma_{2}(p_2)=p_2+(\epsilon_2-\epsilon_3)p_1$. Moreover, $\sigma_{2}$ acts trivially on the classes $p_1,\ p_3,\ p_4$. We will see that this is not a coincidence.
\end{ex}
\section{Permutation Action and KT-classes}\label{sec:perm-action-KT-classes}
In this section we study the behaviour of the KT-classes under the $\mathfrak{S}_{\underline{k}}$-action introduced in  the previous section. This allows us to prove that the equivariant cohomology is isomorphic to a direct sum of trivial representations (as graded twisted $R$-module). In \S~\ref{subsec:perm-action-cell-closures}, we generalise this to the action on certain cell closures inside the quiver Grassmannian.

\begin{rem}
If a moment graph $\mathcal{G}(X,T,\chi)$ is fixed, the relation $\succeq_\chi$ is uniquely determined, so that we will spare notation, by dropping the index, and write $\succeq$. Moreover, once $\mathcal{G}(X,T,\chi)$ is given, it is clear that any KT-class is intended w.r.t. $\succeq=\succeq_\chi$ and we will avoid to repeat it in any statement.
\end{rem}

\begin{lem}Let $M \in \mr{rep}_\C(\Delta_n)$ be homogeneous and nilpotent. Let $X=\mr{Gr}_\mb{e}(M)$ and  $\mathcal{G}=\mathcal{G}(X,T,\chi)$.  Let $\check{E}: L\stackrel{\alpha_{\check{E}}}{\to} L'\in\mathcal{G}_1$ such that $\#\mathcal{G}^{\partial L'}_1=\# \mathcal{G}^{\partial L}_1-1$. If $p^{L'}\in H^\bullet_T(X)$ is a KT-class, then
\begin{equation}\label{eqn:KTClassAdjacent}
p^{L'}_{L}=\prod_{E\in\mathcal{G}^{\partial L}_1\setminus\{L\to L'\}}\alpha_E.
\end{equation}
\end{lem}
\begin{proof}
By Lemma \ref{lem:EdgesAdjacent} we know that $F\in \mathcal{G}_1^{\partial L}$ if and only if there exists a (unique) $E:N\rightarrow L'\in\mathcal{G}_1^{\partial L'}\setminus \{\check{E}\}$ such that  $F=\mu_{\check{E}}(E)=(\mu_{\check{E}}N\rightarrow L)$. Moreover, by the PS-property,  $\#\mathcal{G}_1^{\partial\mu_{\check{E}}N}<\#\mathcal{G}_1^{\partial L}=\mathcal{G}_{\partial L'}-1$ and we deduce (again by the PS-property) that $\mu_{\check{E}}N\not\succ L'$. Thus, by (KT3), we have $p^{L'}_{\mu_{\check{E}}N}=0$ for any such $N$, and hence
\[
p^{L'}_L\equiv 0 \mod  \alpha_{\mu_{\check{E}}(E)}, \quad \hbox{ for all }E\in \mathcal{G}_1^{\partial L'}\setminus \{\check{E}\}.
\]
Since the characters $\alpha_{\mu_{\check{E}}(E)}$, with $E\in \mathcal{G}_1^{\partial L'}\setminus \{\check{E}\}$, are pairwise linearly independent, is must hold
\[
\prod_{E\in \mathcal{G}_1^{\partial L'}\setminus \{\check{E}\}}\alpha_{\mu_{\check{E}}(E)} \hbox{ divides }p^{L'}_L.
\]
Moreover, by (KT2), $p^{L'}_L$ is homogeneous and
\[
\textrm{deg}(p^{L'}_L)=\textrm{deg}\left(\prod_{E\in \mathcal{G}_1^{\partial L'}\setminus \{\check{E}\}}\alpha_{\mu_{\check{E}}(E)}\right)
\]
and hence there exists a constant $z\in \mathbb{C}$ such that 
\[
p^{L'}_L=z\left( \prod_{E\in \mathcal{G}_1^{\partial L'}\setminus \{\check{E}\}}\alpha_{\mu_{\check{E}}(E)} \right)
\]
The constant $z$ is uniquely determined by imposing the condition  $p^{L'}_{L}\equiv p^{L'}_{L'} \mod \alpha_{\check{E}}$. Indeed, by Lemma \ref{lem:EdgesAdjacent} we know that  
\[
\left(\prod_{E\in \mathcal{G}_1^{\partial L'}\setminus \{\check{E}\}}\alpha_{\mu_{\check{E}}(E)}\right)\equiv \left(\prod_{E\in \mathcal{G}_1^{\partial L'}\setminus \{\check{E}\}}\alpha_{E}\right) \mod \alpha_{\check{E}}
\]
and this implies that $z=1$.
\end{proof}
\begin{cor}Let $M \in \mr{rep}_\C(\Delta_n)$ be homogeneous and nilpotent. Let $X=\mr{Gr}_\mb{e}(M)$ and  $\mathcal{G}=\mathcal{G}(X,T,\chi)$. Let $L\in X^T$ and $i\in[d]\setminus\underline{k'}$ be such that $\sigma_i L\prec L$. Let $p^L,p^{\sigma_i L}$ be the corresponding KT-classes. Then,
\begin{equation}\label{eqn:corKT_simplerefln}
p^{\sigma_i L}_{\sigma_i L}=\sigma_i(p_L^{\sigma_i L})=\frac{\sigma_i(p^L_L)}{\epsilon_{i}-\epsilon_{i+1}}.
\end{equation}
\end{cor}
\begin{proof}
By Lemma~\ref{lem:LengthSimpleRefln} we can apply  Lemma~\ref{lem:EdgesAdjacentSimpleRefln} and obtain
\[
\left\{\alpha_E\mid E\in\mathcal{G}_1^{\partial \sigma_i L}\right\}=\Big\{\sigma_i(\alpha_F)\mid F\in\mathcal{G}_1^{\partial L}\setminus\{L\rightarrow\sigma_i L\}\Big\}.
\]
Then the claim follows immediately from the previous lemma with $L' = \sigma_i L$.
\end{proof}

The following is the key result to describe the $\mathfrak{S}_{\underline{k}}$-module structure on $H_T^\bullet(X)$.
\begin{prop}\label{prop:ActionSimpleReflnKTClass}
Let $M \in \mr{rep}_\C(\Delta_n)$ be homogeneous and nilpotent. Let $X=\mr{Gr}_\mb{e}(M)$ and  $\mathcal{G}=\mathcal{G}(X,T,\chi)$. Let $\{p^L\}_{L\in X^T}$ be the unique KT-basis. If $i\in[d]\setminus\underline{k'}$, then
\begin{equation}\label{eqn:sigma_i-action}
\sigma_i\cdot p^L=
\left\{
\begin{array}{ll}
p^L+(\epsilon_i-\epsilon_{i+1})p^{\sigma_i L}&\hbox{if }\sigma_i L\prec L,\\
p^L&\hbox{otherwise.}
\end{array}
\right.
\end{equation}
\end{prop}
\begin{proof}
For convenience, let us denote $\mathcal{G}:=\mathcal{G}(X,T,\chi)$ and $q:=\sigma_i\cdot p^L$, that is $q_N=\sigma_i(p^L_{\sigma_i N})$ for any $N\in X^T$.
Since $\{p^L\}_{L\in X^T}$ is an $R$-basis, there exist (uniquely determined) homogeneous polynomials $c_N\in R$ such that
\begin{equation}\label{eqn:q-expansion-in-basis}
q=\sum_{N\in X^T} c_N p^N.
\end{equation}
Hence we have to show that the polynomials $c_N$ are of the form as claimed in (\ref{eqn:sigma_i-action}).

By the linearity of the $\mathfrak{S}_{\underline{k}}$-action and by (KT2) we obtain
\[
\deg(q)=\deg(p^L)=\#\mathcal{G}_1^{\partial L}=\deg(c_N)+\deg(p^N)=\deg(c_N)+\#\mathcal{G}_1^{\partial N},
\] 
and we deduce that $c_N\neq 0$ only if $\#\mathcal{G}_1^{\partial N}\leq \#\mathcal{G}_1^{\partial L}$.

Secondly, assume that $N$ is minimal (with respect to the partial order $\preceq$) such that $q_N\neq 0$. If this is the case, also $p_{\sigma_i N}^L\neq 0$ and hence $\sigma_i N\succeq L$ by (KT3), and by the Palais-Smale property $\#\mathcal{G}_1^{\partial \sigma_i N}\geq \#\mathcal{G}_1^{\partial L}$. On the other hand, since $N$ is minimal, it follows from (KT3) that $c_N \neq 0$ and hence $\#\mathcal{G}_1^{\partial N}\leq \#\mathcal{G}_1^{\partial L}$. Therefore $\#\mathcal{G}_1^{\partial \sigma_i N}\geq \#\mathcal{G}_1^{\partial L} \geq \#\mathcal{G}_1^{\partial N}$ so that the Palais-Smale property implies $N\preceq \sigma_i N$ since  $N$ and $\sigma_iN$ are always comparable. If we assume additionally that $\sigma_i N\neq L$, we have $N\prec \sigma_i N$, and $\#\mathcal{G}_1^{\partial N}=\#\mathcal{G}_1^{\partial \sigma_i N}-1$ holds by Lemma~\ref{lem:LengthSimpleRefln}.
Therefore, if $N\neq L,\sigma_i L$, 
\[
\#\mathcal{G}_1^{\partial L}\geq \#\mathcal{G}_1^{\partial N}=\#\mathcal{G}_1^{\partial \sigma_i N}-1\geq \#\mathcal{G}_1^{\partial L}
\]
and we conclude $\#\mathcal{G}_1^{\partial N}=\#\mathcal{G}_1^{\partial L}$. 

Combining the above computations the expansion (\ref{eqn:q-expansion-in-basis}) reduces to
\begin{equation}\label{eqn:sumKTclasses}
q=
\left\{
\begin{array}{ll}
c_L p^L+c_{\sigma_i L} p^{\sigma_i L}+\sum_{\substack{N\neq L,\sigma_i L\\
L,N\prec \sigma_i N\\ \#\mathcal{G}_1^{\partial N}=\#\mathcal{G}_1^{\partial L}}}c_N p^N, &\hbox{if }\sigma_iL\prec L,\\
\\
c_L p^L+\sum_{\substack{N\neq L,\sigma_i L\\
L,N\prec \sigma_i N\\ \#\mathcal{G}_1^{\partial N}=\#\mathcal{G}_1^{\partial L}}}c_N p^N, &\hbox{otherwise}.
\end{array}
\right. 
\end{equation}
Now, we have to distinguish the two cases $\sigma_i L\prec L$ and $\sigma_iL\not\prec L$.

For {\bf $\sigma_i L\prec L$}, to obtain the claim, we have to verify that $c_L=1$, $c_{\sigma_i L}= (\epsilon_i-\epsilon_{i+1})$ and $c_N=0$ for all other $N \in X^T$, as introduced in \eqref{eqn:q-expansion-in-basis}. By the definition of $q_L$ and (KT3)
\[
q_L=\sigma_i(p^L_{\sigma_i L})=0,
\]
and by definition of $q_{\sigma_i L}$ and \eqref{eqn:corKT_simplerefln}
\[
q_{\sigma_i L}=\sigma_i(p^L_L)=(\epsilon_{i}-\epsilon_{i+1})p^{\sigma_i L}_{\sigma_i L}.
\]
We show now that this implies $c_L=1$, $c_{\sigma_i L} = (\epsilon_i-\epsilon_{i+1})$.
Since $q=\sigma_i\cdot p^L$, we can apply \eqref{eqn:corKT_simplerefln} and consider the $L$-coordinate to get
\[ q_L = p_L^L+(\epsilon_ {i}-\epsilon_{i+1})p_L^{\sigma_i L}. \]
On the other hand, by \eqref{eqn:corKT_simplerefln}, we have 
\[
p_L^L=\sigma_i\big((\epsilon_i-\epsilon_{i+1}) \sigma_i(p^{\sigma_i L}_{L})\big)=(\epsilon_{i+1}-\epsilon_i) p^{\sigma_i L}_{L},
\]
and hence $q_L=0$. By looking at the $\sigma_iL$-coordinate, it follows again from \eqref{eqn:corKT_simplerefln}  
 that
\[
q_{\sigma_i L} = p^L_{\sigma_i L}+(\epsilon_i-\epsilon_{i+1})p^{\sigma_i L}_{\sigma_i L}=(\epsilon_i-\epsilon_{i+1})p^{\sigma_i L}_{\sigma_i L}
\]
since $p^L_{\sigma_i L}=0$ by (KT3). Thus $q_N=p^L_N+(\epsilon_i-\epsilon_{i+1})p^{\sigma_i L}_N$ holds for $N=L,\sigma_i L$.

To conclude the study of the case $\sigma_i L\prec L$, it is left to show that $q_N=p^L_N+(\epsilon_{i}-\epsilon_{i+1})p^{\sigma_i L}_N$ for any $N\neq L,\sigma_i L$ such that $L,N\prec \sigma_i N$ and $\#\mathcal{G}_1^{\partial N}=\#\mathcal{G}_1^{\partial L}$. Note that the condition on $\#\mathcal{G}^{\partial N}_1$ implies $N\not\succ L$ and so $p^L_N= 0$, as $N\neq L$, and the desired equality reduces to $q_N = (\epsilon_i-\epsilon_{i+1})p^{\sigma_iL}_N$.

Recall that $\sigma_i N \succeq L$. It follows immediately from $\#\mathcal{G}_1^{\partial\sigma_i N}=\#\mathcal{G}^{\partial N}+1=\#\mathcal{G}^{\partial L}+1$ and the Palais-Smale property that $\sigma_i N\rightarrow L\in\mathcal{G}_1$, and so also $N\rightarrow\sigma_i L\in\mathcal{G}_1$. Then by \eqref{eqn:KTClassAdjacent} we have
\[
p^L_{\sigma_i N}=\prod_{E\in\mathcal{G}_1^{\partial \sigma_i N}\setminus\{\sigma_i N\rightarrow L\}}\!\!\alpha_E,\qquad\qquad p^{\sigma_i L}_N=\prod_{F\in\mathcal{G}_1^{\partial  N}\setminus\{ N\rightarrow\sigma_i L\}}\!\!\alpha_F.
\]
By Lemma \ref{lem:EdgesAdjacentSimpleRefln},
\[
\left\{\alpha_E\mid E\in\mathcal{G}_1^{\partial \sigma_i N}\right\}=\Big\{\sigma_i(\alpha_F)\mid F\in\mathcal{G}_1^{\partial N}\Big\}\cup\Big\{\epsilon_{i+1}-\epsilon_i\Big\}
\]
so that
\begin{align*}
q_N=\sigma_i(p_{\sigma_iN}^L)&=\prod_{ E\in\mathcal{G}_1^{\partial \sigma_i N}\setminus\{\sigma_i N\rightarrow L\}}\sigma_i(\alpha_E)\\
&=\sigma_i(\epsilon_{i+1}-\epsilon_i)\cdot \prod_{F\in\mathcal{G}_1^{\partial  N}\setminus\{ N\rightarrow \sigma_iL\}}\alpha_F\\
&=(\epsilon_i-\epsilon_{i+1})p^{\sigma_i L}_N.
\end{align*}
Finally, let $\sigma_iL\not\prec L$. We claim that
\[
\left(\#\mathcal{G}_1^{\partial N}=\#\mathcal{G}_1^{\partial L} \hbox{ and }q_N\neq 0\right)\quad\Rightarrow\quad N=L.
\]
From this it follows by (KT1) and (KT3), that $c_N=0$ for any $N\neq L$ and $c_L=1$ solves \eqref{eqn:sumKTclasses}, as desired. 

Assume by contradiction that $q_N\neq 0$ and $N\neq L$. Recall that we have an arrow $\sigma_i N\rightarrow L\in\mathcal{G}_1$, and hence, by \eqref{eqn:SkActionGraph}, either $N\rightarrow \sigma_i L\in\mathcal{G}_1$ or $\sigma_i L\rightarrow N\in\mathcal{G}_1$. If $\sigma_i L=L$, the existence of such an edge contradicts the Palais-Smale property, as $\#\mathcal{G}_1^{\partial N}=\#\mathcal{G}_1^{\partial L}$ implies that $L=\sigma_i L$ is not comparable with $N$. If, instead, $\sigma_iL\succ L$, then by Lemma~\ref{lem:LengthSimpleRefln} we have $\#\mathcal{G}_1^{\partial \sigma_i L}=\#\mathcal{G}_1^{\partial L}+1$ and so
\[
\#\mathcal{G}_1^{\partial \sigma_i L}=\#\mathcal{G}_1^{\partial L}+1>\#\mathcal{G}_1^{\partial L}=\#\mathcal{G}_1^{\partial N}
\]
and by the Palais-Smale property we have that $\sigma_i L\rightarrow N$. By the proof of \eqref{eqn:KTClassAdjacent}, we have that 
\[
\mathcal{G}_1^{\partial \sigma_i L}=\sigma_i(\mathcal{G}_1^{\partial L})\cup\{\sigma_i L\rightarrow L\}
\]
and, since $N\neq L$, we must have $\sigma_i(\sigma_i L\rightarrow N)=L\rightarrow\sigma_i N\in\mathcal{G}_1$, but this contradicts $\sigma_i N\succ L$.
\end{proof}

\begin{lem}\label{lem:SkActionKTClass}(cf. \cite[Lemma 3.6]{Tymoczko2008}) Let $M \in \mr{rep}_\C(\Delta_n)$ be homogeneous and nilpotent. Let $X=\mr{Gr}_\mb{e}(M)$ and  $\mathcal{G}=\mathcal{G}(X,T,\chi)$. Let $\{p^L\}_{L\in X^T}$ be the unique KT-basis. Then, for any $\sigma\in\mathfrak{S}_{\underline{k}}$ and $L\in X^T$ 
\[
\sigma\cdot p^L=p^L+\sum_{N\prec L}c^{L,\sigma}_N p^N,
\]
where $c_N^{L,\sigma}$ is a homogeneous polynomial of degree $\#\mathcal{G}_1^{\partial L}-\#\mathcal{G}_1^{\partial N}$. 
\end{lem}
\begin{proof}
The proof is by induction on the Coxeter length $\ell(\sigma)$ of $\sigma$. If $\ell(\sigma)=0$, then $\sigma$ is just the identity and there is nothing to show. The case $\ell(\sigma)=1$ is Proposition \ref{prop:ActionSimpleReflnKTClass}. Otherwise, $\ell(\sigma)=l\geq 2$ and there exist simple reflections $\sigma_{i_1}\sigma_{i_2}, \ldots, \sigma
_{i_l}\in\mathfrak{S}_{\underline{k}}$ such that $\sigma=\sigma_{i_1}\sigma_{i_2}\ldots\sigma_{i_l}$. We set  $\sigma':=\sigma_{i_2}\ldots\sigma_{i_l}$ and get
by induction 
\[
\sigma'\cdot p^L=p^L+\sum_{N\prec L}c^{L,\sigma'}_N p^ N,
\]
Hence, by the linearity of $\sigma_{i_1}$ and by Proposition \ref{prop:ActionSimpleReflnKTClass} we obtain
\begin{align*}
\sigma\cdot p^L&=\sigma_{i_1}(\sigma'\cdot p^L)\\
&=\sigma_{i_1}(p^L)+\sum_{N\prec L}c^{L,\sigma'}_N \sigma_{i_1}(p^ N)\\
&=\sigma_{i_1}(p^L)+\sum_{\substack{N\prec L\\
\sigma_{i_1}N\prec N}}c^{L,\sigma'}_N (p^ N+(\epsilon_{i_1}-\epsilon_{i_1+1})p^{\sigma_{i_1} N})
+\sum_{\substack{N\prec L\\
\sigma_{i_1}N\not\prec N}}c^{L,\sigma'}_N p^ N\\
&=p^L+\sum_{\substack{N\prec L\\ \sigma_{i_1} N\preceq  N}} c^{L,\sigma'}_N p^N+\sum_{\substack{N\prec \sigma_iN\\ N\prec L}}\big(c_N^{L,\sigma'}+(\epsilon_{i_1}-\epsilon_{i_1+1})c_{\sigma_i N}^{L,\sigma'}\big)p^N,
\end{align*}
where we set by convention $c_L^{L,\sigma'}=1$ and $c_N^{L,\sigma'}=0$ if $N\not\preceq L$. The claim about the homogeneity and degree of the polynomials $c_N^{L,\sigma}$ follows, by induction, from the previous formula.
\end{proof}

Note that the ring $R$ is $\mathbb{Z}$-graded, and this grading induces a grading on each free $R$-module. We will write $\rm Triv^d$ for the rank one $R$-module concentrated in degree $2d$, considered as a $\mathfrak{S}_{\underline{k}}$-representation in the obvious way.

The following result extends \cite[Theorem 3.8]{Tymoczko2008}.
\begin{thm}\label{thm:mainThmTrivialRepn}Let $M \in \mr{rep}_\C(\Delta_n)$ be homogeneous and nilpotent. Let $X=\mr{Gr}_\mb{e}(M)$ and  $\mathcal{G}=\mathcal{G}(X,T,\chi)$. Let $\{p^L\}_{L\in X^T}$ be the unique KT-basis. Then the $\mathfrak{S}_{\underline{k}}$-representation $H_T^\bullet(X)$ is isomorphic to $\bigoplus_{L\in X^T}{\rm Triv}^{\#\mathcal{G}_1^{\partial L}}$ as a graded twisted $R$-module.
\end{thm}
\begin{proof}
To prove the claim we show that there exists an $R$-basis $\{h^L\}_{L\in X^T}$ of $H_T^\bullet(X)$ having the property that $\sigma\cdot h_L=h_L$ for any $\sigma\in\mathfrak{S}_{\underline{k}}$ and any $L\in X^T$.

We define $h^L:=\frac{1}{k_1!k_2!\cdot\ldots\cdot k_{d_0}!}\sum_{\sigma\in\mathfrak{S}_{\underline{k}}}\sigma\cdot p^L $ and observe that by Lemma \ref{lem:SkActionKTClass} we have
\begin{align*}
h^L&=\frac{1}{k_1!k_2!\cdot\ldots\cdot k_{d_0}!}\sum_{\sigma\in\mathfrak{S}_{\underline{k}}}(p^L+\sum_{N\prec L}c_N^{L,\sigma} p^N)\\ 
&=p^L+\frac{1}{k_1!k_2!\cdot\ldots\cdot k_{d_0}!}\sum_{N\prec L}\left(\sum_{\sigma\in\mathfrak{S}_{\underline{k}}}c_N^{L,\sigma}\right) p^N.
\end{align*}
Therefore, if we write $h^L=\sum_{N\in X^T} d_N^L p^N$, we always have $d_L^L=1$ and $d_N^L=0$ unless $N\preceq L$.

To conclude, pick a total order on $X^T$ refining $\prec$.  By the previous consideration we see that the matrix with entry $(N,L)$ given by $d_N^L$ (and columns/rows ordered according to the chosen total order) is upper triangular with ones on the diagonal, and so $\{h^L\}_{L\in X^T}$ is a basis too.
\end{proof}
Since the action on $R$ is twisted, the $\mathfrak{S}_{\underline{k}}$-action on $H_T(X)$ as a $\mathbb{C}$-vector space is not trivial. Denote by $P$ the $\mathbb{C}$-vector space $R$ with the $\mathfrak{S}_{\underline{k}}$-representation structure we have been working with (the one induced by the $\mathfrak{S}_{\underline{k}}$-action on $T$). In the following statement ${\rm Triv}^d$ denotes a one-dimensional $\mathbb{Z}$-graded $\mathbb{C}$-vector space concentrated in degree $d$ equipped with trivial $\mathfrak{S}_{\underline{k}}$-action.
\begin{cor}Under the hypotheses of the previous theorem, 
\[H_T^\bullet(X)\cong\bigoplus_{L\in X^T}{\rm Triv}^{\#\mathcal{G}_1^{\partial L}}\otimes P\]
as a graded $\mathfrak{S}_{\underline{k}}$-module.
\end{cor}
The following corollary deals with the induced permutation action on usual cohomology.
\begin{cor}\label{cor:Cor2mainThm}Under the same hypotheses as Theorem \ref{thm:mainThmTrivialRepn}, 
\[H^\bullet(X)\cong\bigoplus_{L\in X^T}{\rm Triv}^{\#\mathcal{G}_1^{\partial L}}.\]
\end{cor}
\begin{proof}
By equivariant formality (see \cite[Proposition 1]{Brion2000}), we have a surjective homomorphism  
\[
H^\bullet_T(X)\stackrel{\pi}{\rightarrow} H^\bullet(X)
\]
which send the KT-basis $\{p^L\}_{L\in X^T}$ to the $\mathbb{C}$-basis $\{\pi(p^L)\}_{L\in X^T}$ and, for any collection of $(a_L)\in\bigoplus_{L\in X^T} R$,
\[
\pi\left(\sum_{L\in X^T} a_L p^L\right)=\sum_{L\in X^T}\overline{a_L} \ \pi(p^L),\]
 where $\overline{a_L}\in\mathbb{C}$ is the image of the surjective map $R\rightarrow \mathbb{C}$ whose kernel is generated by the homogeneous elements of strictly positive degree. 
\end{proof}
\subsection{Permutation Action on certain Cell Closures}\label{subsec:perm-action-cell-closures}
In this section we extend our previous results on the $\mathfrak{S}_{\underline{k}}$-action, to cell closures which are union of smaller cells.
\begin{rem}
If $\mathcal{G}=\mathcal{G}(X,T,\chi)$ is not Palais-Smale, there are counter examples for the property that every cell closure is a union of cells (see \cite[Example~7.6]{LaPu2020} and \cite[Example~4]{CFFFR2017}). At the moment, it is unclear to us if a Palais-Smale orientation of $\mathcal{G}$ is sufficient to deduce this property.
\end{rem}

\begin{lem}
Let $M \in \mr{rep}_\C(\Delta_n)$ be homogeneous and nilpotent. Let $L\in X^T$ and assume that $\overline{W_L}=\bigcup_{N\preceq L} W_N$. Then also $(\overline{W_L}, T)$ is a BB-filterable GKM-variety. Moreover, if $\mathcal{G}=\mathcal{G}(X,T,\chi)$ is Palais-Smale, the same holds for its full subgraph $\mathcal{G}_{\preceq L}$ whose  vertex set is $\{N\mid N\preceq L\}$.
\end{lem}
\begin{proof}
The first statement is obvious, as $\overline{W_L}$ is a closed $T$-stable subvariety of a GKM-variety with a stratification by affine spaces. The Palais-Smale property follows from the fact that for any $N\preceq L$ if $N'\in X^T$ is such that $N\rightarrow N'\in \mathcal{G}_1$, then $N'\preceq N\preceq L$. Thus, $\mathcal{G}_1^{\partial N}\subseteq (\mathcal{G}_{\preceq L})_1$ for any $N\preceq L$. We conclude that if $N\rightarrow N'\in(\mathcal{G}_{\preceq L})_1$ then $\sharp \mathcal{G}_1^{\partial N}> \sharp \mathcal{G}_1^{\partial N'}$ if $\mathcal{G}$ is itself Palais-Smale.
\end{proof}

From the above lemma it follows that we can apply GKM-theory also to $\overline{W_L}$. We can hence consider the following homomorphism of graded $R$-modules, called for obvious reasons the restriction map:
\[
\iota^*:H_T^\bullet(X)\rightarrow H_T^\bullet(\overline{W_L}), \quad (f_N)_{N\in X^T}\mapsto (f_N)_{\substack{N\in X^T:\\N\preceq L}}
\] 

The following is the key ingredient to carry over our results on $X$ into cell closures $\overline{W_L}$ which are union of smaller cells, and it extends \cite[Lemma 3.3]{Tymoczko2008}.
\begin{lem}\label{lem:RestrictionKTClasses}Let $M \in \mr{rep}_\C(\Delta_n)$ be homogeneous and nilpotent. Let $X=\mr{Gr}_\mb{e}(M)$ and let $L\in X^T$ be a fixed point and assume that $\overline{W_L}=\bigcup_{N\preceq L} W_N$. Let  $\mathcal{G}=\mathcal{G}(X,T,\chi)$.  Then $p^{N,(L)}:=\iota^*(p^N)$ is the unique KT-class for $N$ in $H_T^\bullet(\overline{W_L})$.
\end{lem}
\begin{proof}We have to verify that $p^{N,(L)}$ satisfies (KT1),(KT2) and (KT3).

Firstly, (KT1) holds obviously, as $\mathcal{G}^{\partial N}_1$ is entirely contained in $\mathcal{G}_{\preceq L}=\mathcal{G}(\overline{W_L},T,\chi)$ and hence 
\[
p^{N,(L)}_N=\iota^*(p^N)=p^N_N= \prod_{\substack{ E\in \mathcal{G}_1^{\partial N}}} \alpha_E =\prod_{\substack{ E\in (\mathcal{G}_{\preceq L})_1^{\partial N}}} \alpha_E.\]

Secondly,  (KT2) follows immediately from the fact that the restriction map preserves the degree.

To conclude, it is left to check (K3). Let $N'\in (\mathcal{G}_{\preceq L})_0 = \overline{W_L}^T$ be such that $N'\not\succeq N$. Since $\mathcal{G}_{\preceq L}$ is a full subgraph of $\mathcal{G}$, the partial order on $\overline{W_L}^T$ is induced by the partial order on $X^T$, and hence if $N'\not\succeq N$ in $\overline{W_L}^T$, also $N'\not\succeq N$ in $X^T$. Thus, $p^N_{N'}=\iota^*(p^N)_{N'}=p^N_{N'}=0$, as desired.
\end{proof}

\begin{thm}Let $M \in \mr{rep}_\C(\Delta_n)$ be homogeneous and nilpotent. Let $X=\mr{Gr}_\mb{e}(M)$ and  $\mathcal{G}=\mathcal{G}(X,T,\chi)$. Let $L\in X^T$ be a fixed point and assume that $\overline{W_L}=\bigcup_{N\preceq L} W_N$. Then
\begin{enumerate}
\item there is an $\mathfrak{S}_{\underline{k}}$-action on $H_T^\bullet(\overline{W_L})$ given by
\begin{equation*}
\sigma_i\cdot p^{N,(L)}=
\left\{
\begin{array}{ll}
p^{N,(L)}+(\epsilon_i-\epsilon_{i+1})p^{\sigma_i N,(L)}&\hbox{if }\sigma_i N\prec N,\\
p^{N,(L)}&\hbox{otherwise,}
\end{array}
\right. \qquad (N\preceq L)
\end{equation*} 
for any $i\in[d]\setminus\underline{k'}$.
\item the $\mathfrak{S}_{\underline{k}}$-representation $H_T^\bullet(\overline{W_L})$ is isomorphic to $\bigoplus_{N\in \overline{W_L}^T}{\rm Triv}^{\#\mathcal{G}_1^{\partial N}}$ as a graded twisted $R$-module,
\item \(H_T^\bullet(\overline{W_L})\cong\bigoplus_{N\in \overline{W_L}^T}{\rm Triv}^{\#\mathcal{G}_1^{\partial N}}\otimes P\)
as a graded $\mathfrak{S}_{\underline{k}}$-module,
\item \(H^\bullet(\overline{W_L})\cong\bigoplus_{N\in \overline{W_L}}{\rm Triv}^{\#\mathcal{G}_1^{\partial N}}\) as a graded $\mathfrak{S}_{\underline{k}}$-module.
\end{enumerate}
\end{thm}
\begin{proof}
\begin{enumerate}
\item 
 Let $\sigma\in\mathfrak{S}_{\underline{k}}$. By Lemma \ref{lem:RestrictionKTClasses}, any element of $H_T^*(X)$ can be written uniquely as $\sum_{N\preceq L}c^L_N \iota^*(p^N)$ for some $a^L_N\in R$. We set
\[
\sigma\left(\sum_{N\preceq L}a^L_N \iota^*(p^N)\right)=\sum_{N\preceq L }\sigma(a^L_N) \iota^*(\sigma\cdot p^N).
\]
This is a well-defined action thanks to Lemma \ref{lem:SkActionKTClass}. Moreover, such an action is nothing but the restriction of the $\mathfrak{S}_{\underline{k}}$-action. Hence, the claimed formula follows from Proposition~\ref{prop:ActionSimpleReflnKTClass}
\item By Lemma~\ref{lem:SkActionKTClass}, we see that we can define $h_N$, for $N\preceq L$, as in Theorem \ref{thm:mainThmTrivialRepn}. By the proof of Theorem~\ref{thm:mainThmTrivialRepn}, every $h_N$ is $\mathfrak{S}_{\underline{k}}$-invariant and the set $\{h_N\}_{N\preceq L}$ is an $R$-basis of $H_T^\bullet(\overline{W_L})$. This implies the claim.
\item The claim follows immediately from the previous point.
\item Since $\overline{W_L}$ is equivariantly formal (see \cite[Proposition 1]{Brion2000}), there is a surjective map
\[
\pi:H_T^\bullet(\overline{W_L})\rightarrow H^\bullet(\overline{W_L}),
\]
the set $\{\pi(p^N)\}_{N\preceq L}$ is a $\mathbb{C}$-basis and the same proof of Corollary \ref{cor:Cor2mainThm} goes through.
\end{enumerate}\end{proof}
\begin{rem}If $X$ is a flag variety, then every cell closure is a Schubert variety and is a union of (Schubert) cells. The above theorem generalises the main results of \cite{Tymoczko2008}.
\end{rem}
\begin{rem}\label{rem:permutation-action-general-setting}
The proofs in Section~\ref{sec:Permutation-Action} are based on the structure of the moment graph as described in Theorem~\ref{trm:comb-moment-graph} and the fact that every quiver Grassmannian for a homogeneous nilpotent $\Delta_n$-representation is Palais-Smale (see Proposition~\ref{prop:homg-forest-are-P-S}). We expect that once found an appropriate $T$-action, a permutation representation on more general quiver Grassmannians can be defined and investigated in an analogous way.
\end{rem}
\section{Divided Difference Operators}\label{sec:DividedDifferenceOps}
Divided difference operators have a long history, starting with Newton's work, as explained in \cite{lascoux}. In the setting of Schubert calculus they were firstly applied by Bernstein-Gel'fand-Gel'fand \cite{bgg} and Demazure \cite{demazure} about fifty years ago. In this section we generalise Tymoczko's definition of (left) divided different operators on the equivariant cohomology of the flag variety (cf.\cite[\S3.4]{Tymoczko2008}) to our situation. Right divided difference operators were generalised in \cite{laninizainoulline} from a completely different perspective, by using the formalism of moment graph fibrations and structure algebras, which seems not to be applicable in general to the moment graph from Theorem~\ref{trm:comb-moment-graph}. 

For $i\in[d]\setminus\underline{k'}$, we denote by $\partial_i:R\rightarrow R$ the (lowering degree) $\mathbb{C}$-linear operator given by $\partial_i(a)=\frac{a-\sigma_i(a)}{(\epsilon_{i+1}-\epsilon_{i})}$. The following lemma is the key to define divided difference operators.
\begin{lma}
Let $M \in \mr{rep}_\C(\Delta_n)$ be a nilpotent representation. If $i\in[d]\setminus\underline{k'}$ and $f \in H^\bullet_T(X)$, then
\[ \frac{f-\sigma_i\cdot f}{\epsilon_{i+1}-\epsilon_{i}} \in H^\bullet_T(X). \]
\end{lma}
\begin{proof}
If $i$ is as above, then $\sigma_i\in\mathfrak{S}_{\underline{k}}$. Let $L\in X^T$ be such that $\sigma_i L\neq L$, then by Lemma~\ref{lem:reflectionEdge}, we have $\sigma_i L\stackrel{\epsilon_{i+1}-\epsilon_i}{-\!\!\!\!-\!\!\!\!-\!\!\!\!-\!\!\!\!-} L\in\overline{\mathcal{G}(X,T)}_1$. Therefore, if $f=(f_L)_{L\in X^T}\in H^\bullet_T(X)$, then
$f_L-f_{\sigma_i L}\in(\epsilon_i-\epsilon_{i+1})R$ and also $f_{\sigma_i L}-\sigma_i(f_{\sigma_iL})\in (\epsilon_{i+1}-\epsilon_{i})R$. We deduce that $f_L-\sigma_i(f_{\sigma_i L})\in(\epsilon_{i+1}-\epsilon_{i})R$, that is $\frac{f_L-\sigma_i(f_{\sigma_i L})}{\epsilon_{i+1}-\epsilon_{i}}\in R$. Therefore, $q:=\frac{f-\sigma_i\cdot f}{\epsilon_{i+1}-\epsilon_{i}}\in\bigoplus_{L\in X^T} R$. Moreover for any edge $E:L\stackrel{\alpha}{\rightarrow} L'\in \overline{\mathcal{G}(X,T)}_1$ also $F:\sigma_i L\stackrel{\sigma_i(\alpha)}{-\!\!\!\!-\!\!\!\-\!\!\!\!-\!\!\!\!-}L'\in \overline{\mathcal{G}(X,T)}_1$ and  we have
\[
q_L-q_{L'}=\frac{f_L-\sigma_i( f_{\sigma_iL})}{\epsilon_{i+1}-\epsilon_{i}}-\frac{f_{L'}-\sigma_i\cdot f_{\sigma_iL'}}{\epsilon_{i+1}-\epsilon_{i}}=\frac{\overbrace{(f_L-f_{L'})}^{\in\alpha R}-\overbrace{\sigma_i( f_{\sigma_iL}- f_{\sigma_iL'})}^{\in\alpha R}}{\epsilon_{i+1}-\epsilon_{i}}.
\]
If $L'\neq\sigma_i L$, then 
 $\alpha$ and $\epsilon_{i+1}-\epsilon_i$ are coprime and by the previous computation, being $q_L-q_{L'}\in R$, we deduce that $\alpha(\epsilon_{i+1}-\epsilon_i)\mid q_L-q_{L'}$, and hence $q_L-q_{L'}\in\alpha R$.
 
Otherwise, $L'=\sigma_i L$, then $\alpha=\epsilon_{i+1}-\epsilon_i$ and we set $f_{L'}=f_L-\alpha g$. Thus,
\[
q_L-q_{L'}=\frac{\alpha g-\sigma_i(-\alpha g)}{\alpha}=\frac{\alpha(g-\sigma_i(g))}{\alpha}=g-\sigma_i(g)\in\alpha R.
\]
\end{proof}
\begin{defn}Let $\sigma_i\in\mathfrak{S}_{\underline{k}}$ be a simple reflection. The corresponding divided difference operator on $H^\bullet_T(X)$ is defined as
\[
D_i(f)=\frac{f-\sigma_i\cdot f}{\epsilon_{i+1}-\epsilon_{i}}, \qquad \big(f\in H^\bullet_T(X)\big).
\]
\end{defn}
Notice that to define the above operator we do not need to assume that $\mathcal{G}(X,T,\chi)$ admits a Palais-Smale orientation, or that a KT-basis exists. If $X$ has also this properties, then the divided difference operators are easier to control:
\begin{thm}\label{thm:DividedDiffOpsActionOnKTBasis} Let $M \in \mr{rep}_\C(\Delta_n)$ be homogeneous and nilpotent. Let $X=\mr{Gr}_\mb{e}(M)$ and  $\mathcal{G}=\mathcal{G}(X,T,\chi)$. Let $\{p^L\}_{L\in X^T}$ be the unique KT-basis. Let $i\in[d]\setminus\underline{k'}$ 
 and $f=\sum_{L\in X^T}a_L p^L\in H^\bullet_T(X)$, then
\[
D_i(f)=\sum_{L\in X^T}\partial_i(a_L) p^L+\sum_{\substack{L\in X^T:\\
\sigma_iL\prec L}}\sigma_i(a_L)p^{\sigma_i L}.
\]
\end{thm}
\begin{proof}
By linearity,
\[
D_i(f)=\sum_{L\in X^T}D_i(a_L p^L)=\frac{\sum_{L\in X^T}\big(a_L p^L-\sigma_i(a_L)\ \sigma_i\cdot p^L\big)}{\epsilon_{i+1}-\epsilon_{i}}\\\]
and we can apply Proposition~\ref{prop:ActionSimpleReflnKTClass} to expand $\sigma_i\cdot p^L$ and hence get
\begin{align*}
D_i(f)&=\frac{\sum_{L\in X^T}\left(a_L p^L-\sigma_i(a_L)\ p^L\right)-(\epsilon_i-\epsilon_{i+1})\sum_{\substack{L\in X^T:\\
\sigma_iL\prec L}}\sigma_i(a_L)p^{\sigma_i L}}{\epsilon_{i+1}-\epsilon_{i}}\\
&=\sum_{L\in X^T}\partial_i(a_L) p^L+\sum_{\substack{L\in X^T:\\
\sigma_iL\prec L}}\sigma_i(a_L)p^{\sigma_i L}.
\end{align*}
\end{proof}
The following corollary is an immediate consequence.
\begin{cor}\label{cor:DividedDiffpL}Let $M \in \mr{rep}_\C(\Delta_n)$ be homogeneous and nilpotent. Let $X=\mr{Gr}_\mb{e}(M)$ and  $\mathcal{G}=\mathcal{G}(X,T,\chi)$. Let $L\in X^T$ and let $p^L$ be the corresponding KT-class. Then, for any $i\in[d]\setminus\underline{k'}$,  
\[
D_i(p^L)=\left\{
\begin{array}{ll}
p^{\sigma_i L}     &\hbox{ if }\sigma_iL\prec L,  \\
    0 &\hbox{ otherwise}.
\end{array}\right.
\]
\end{cor}
\subsection{Nil Hecke Algebra Action}
In \cite{KostantKumar1986}, the nil Hecke Algebra was introduced for any Kac-Moody group, and investigated in relation to the cohomology of generalised flag varieties.

We identify the nil Hecke algebra ${}^0\mathcal{H}(\mathfrak{S}_{\underline{k}})$ of the Coxeter group $\mathfrak{S}_{\underline{k}}$ with the subalgebra of $\textrm{End}_{\mathbb{C}}(\mathbb{C}\left[\epsilon_i\mid i\in [d]\right])$ generated by the operators $\partial_i$ with $i\in[d]\setminus\underline{k'}$. The relations that they satisfy are 
\begin{align}\label{eqn:RelnsNH1}
\partial_i^2=0&\qquad i\in[d]\setminus\underline{k'}\\
\label{eqn:RelnsNH2}
\partial_i\partial_j=\partial_j\partial_i & \qquad i,j\in[d]\setminus\underline{k'}\hbox{ and }	[i-j|>1, \\
\label{eqn:RelnsNH3}
\partial_i\partial_{i+1}\partial_i=\partial_{i+1}\partial_i\partial_{i+1} & \qquad i\in[d]\setminus\underline{k'}.
\end{align}

The nil Hecke algebra is $\mathbb{Z}$-graded by declaring 
\[
\textrm{deg}(\partial_i)=-2 
.\]

Thus, as in the classical setting (cf. \cite[Proposition 2.7(c)]{KostantKumar1986}), we obtain an ${}^0\mathcal{H}(\mathfrak{S}_{\underline{k}})$-module structure on $H_T^\bullet(X)$:
\begin{thm}\label{thm:NHRingRepn}Let $M \in \mr{rep}_\C(\Delta_n)$ be a nilpotent representation. Let $X=\mr{Gr}_\mb{e}(M)$. There is a $\mathbb{C}$-linear homomorphism of $\mathbb{Z}$-graded rings
\[
\rho:{}^0\mathcal{H}(\mathfrak{S}_{\underline{k}})\rightarrow\textrm{End}_{\mathbb{C}}(H_T^\bullet(X)), \quad \partial_i\mapsto  D_i\ (i\in [d]\setminus\underline{k'}).
\]
\end{thm}
\begin{proof}
We define $\mathbb{C}$-linear operators $\tilde{D_i}$  ($i\in[d]$) on the direct sum $\bigoplus_{\sigma\in\mathfrak{S}_d}R$, in the same way as the $D_i$'s above.
These $\tilde{D_i}$'s coincide with the divided difference operators from \cite{Tymoczko2009} and we can hence deduce the claim from the analogue claim for $\tilde{D_i}$. In particular, \cite[Proposition 5.6]{Tymoczko2009} gives us relation \eqref{eqn:RelnsNH1}, while \cite[Proposition 5.6]{Tymoczko2009} tells us that divided difference operators satisfy braid relations and hence \eqref{eqn:RelnsNH2} and \eqref{eqn:RelnsNH3} hold. 
\end{proof}
While in the the flag variety case the above representation is faithful (see \cite[Chapter 11]{Kumar2002}), this does not hold in general in our setting, as the following example shows.
\begin{ex} Let us consider the $\Delta_1$-representation $M=A_1^{\oplus 2}\oplus A_2^{\oplus 2}$ (in the notation of Example~\ref{ex:cohomology-generators-loop-quiver}). Let $X=\mr{Gr}_{1}(M)$. We have $[4] \setminus \underline{k'}=\{1,3\}$. It is easy to see that $\rho(\partial_3\partial_1)(p^L)=0$ for any $L\in X^T$, so that $\textrm{Ker}(\rho)\neq(0)$.
\end{ex}
\begin{rem}
It would be interesting to find conditions under which $\rho$ is faithful. This would provide another realisation of the nil Hecke algebra as a subalgebra of the endomorphism algebra of a cohomology ring.
\end{rem}
Differently from the flag variety case, $H_T^\bullet(X)$ is in general not a cyclic module in our setting (see Example \ref{ex:NHModNotCyclic}), while in the (partial) flag variety case every equivariant Schubert class is obtained by applying an appropriate sequence of divided difference operators to the (unique) top degree class (see \cite[Theorem 6.1]{Tymoczko2009}).
 
  
 \begin{ex}\label{ex:NHModNotCyclic}
 Consider the KT-basis from Example \ref{ex:K-T-basis-loop-quiver}. In this case we have $\underline{k}=(1,2)$ and hence $\mathfrak{S}_{\underline{k}}=\langle\sigma_2\rangle$. It is easy to see that $H_T^\bullet(X)$ is not a cyclic ${}^0\mathcal{H}(\mathfrak{S}_{\underline k})$-module. Indeed, let $a_1,a_2,a_3,a_4\in R$, and let $f=a_1p^1+a_2p^2+a_3p^3+a_4p^4$ then 
 \[
 D_2(f)=(\partial_2(a_1)+\sigma_2(a_2))p^1+\partial_2(a_2)p^2+\partial_2(a_3)p^3+\partial_2(a_4)p^4
 \]
 and one sees easily that $p^1\in {}^0\mathcal{H}(\mathfrak{S}_{\underline k})f$ if and only if $a_3,a_2,a_4\in R^{\mathfrak{S}_{\underline{k}}}$. On the other hand, since $T_2$ and multiplication by any complex number do not modify the support of $f$, $p^2\in {}^0\mathcal{H}(\mathfrak{S}_{\underline k})f$ if and only if $a_3=a_4=0$. To conclude we observe that this implies $p^3, p^4\not \in {}^0\mathcal{H}(\mathfrak{S}_{\underline k})f$.
 \end{ex}
 \begin{rem}
 We believe that it is worth it to further investigate the ${}^0\mathcal{H}(\mathfrak{S}_{\underline{k}})$-module structure  on $H_T^\bullet(X)$. For example, it would be interesting to determine conditions under which the latter is a cyclic ${}^0\mathcal{H}(\mathfrak{S}_{\underline{k}})$-module.
 \end{rem}


\end{document}